\pdfoutput=1

\documentclass[none]{imsart}
\usepackage{amsthm,amsmath,amsfonts,bbm}
\usepackage[numbers,sort&compress]{natbib}
\usepackage[colorlinks,citecolor=blue,urlcolor=blue]{hyperref}
\usepackage{url}
\RequirePackage[OT1]{fontenc}
\startlocaldefs

\theoremstyle{plain}
\newtheorem{thm}{Theorem}[section] % reset theorem numbering for each chapter
\newtheorem{lemma}[thm]{Lemma}

\newtheorem{prop}[thm]{Proposition}
\theoremstyle{definition}
\newtheorem{defn}[thm]{Definition}
\newtheorem{assumption}[thm]{Assumption}
\theoremstyle{remark}
\newtheorem*{remark}{Remark}

\renewcommand{\a}{\alpha}
\renewcommand{\b}{\beta}
\newcommand{\e}{\varepsilon}
\renewcommand{\l}{\lambda}
\renewcommand{\o}{\omega}

\newcommand{\E}{\mathbb{E}}

\renewcommand{\L}{\Lambda}
\newcommand{\N}{\mathbb{N}}
\renewcommand{\O}{\Omega}
\renewcommand{\P}{\mathbb{P}}

\newcommand{\R}{\mathbb{R}}

\newcommand{\BB}{\mathcal{B}}
\newcommand{\CC}{\mathcal{C}}

\newcommand{\LL}{\mathcal{L}}
\newcommand{\MM}{\mathcal{M}}
\newcommand{\NN}{\mathcal{N}}
\newcommand{\PP}{\mathcal{P}}

\newcommand{\VV}{\mathcal{V}}

\newcommand{\1}{\mathbbm{1}}

\newcommand{\Ocarre}[3]{\Delta_{#1}(#2,\,#3)}

\newlength{\dhatheight}

\endlocaldefs

\begin{document}

\begin{frontmatter}

% "Title of the paper"
\title{Limit Theorems for Cloning Algorithms}
\runtitle{Limit Theorems for Cloning Algorithms}

% indicate corresponding author with \corref{}
% \author{\fnms{John} \snm{Smith}\corref{}\ead[label=e1]{smith@foo.com}\thanksref{t1}}
% \thankstext{t1}{Thanks to somebody} 
% \address{line 1\\ line 2\\ printead{e1}}
% \affiliation{Some University}

\begin{aug}
\author{\fnms{Letizia} \snm{Angeli}}, 
\author{\fnms{Stefan} \snm{Grosskinsky}}
\and
\author{\fnms{Adam M.} \snm{Johansen}}

%\affiliation{University of Warwick\thanksmark{m1}, The Alan Turing Institute\thanksmark{m2}, TU Delft\thanksmark{m3}, Heriot-Watt University \thanksmark{m4}}

\address{Letizia Angeli: University of Warwick, The Alan Turing Institute and Heriot-Watt University. \\\noindent Email: \href{mailto:l.angeli@hw.ac.uk}{l.angeli@hw.ac.uk}\\\noindent
Stefan Grosskinsky: University of Warwick, The Alan Turing Institute and TU Delft. \\\noindent Email: \href{mailto:s.w.grosskinsky@tudelft.nl}{s.w.grosskninsky@tudelft.nl}\\\noindent
Adam M. Johansen: University of Warwick and The Alan Turing Institute.\\\noindent Email: \href{mailto:a.m.johansen@warwick.ac.uk}{a.m.johansen@warwick.ac.uk}
}
%Department of Statistics\\
%University of Warwick\\
%Coventry, CV4 7AL\\
%United Kingdom\\
%\printead{e1}\\
%\phantom{E-mail:\ }\printead*{e2}\\
%\phantom{E-mail:\ }\printead*{e3}}

\runauthor{L. Angeli, S. Grosskinsky and A. M. Johansen}
\end{aug}

\begin{abstract}

 Large deviations for additive path functionals of stochastic processes have attracted significant research interest, in particular in the context of stochastic particle systems and statistical physics. Efficient numerical `cloning' algorithms have been developed to estimate the scaled cumulant generating function, based on importance sampling via cloning of rare event trajectories. So far, attempts to study the convergence properties of these algorithms in continuous time have led to only partial results for particular cases. Adapting previous results from the literature of particle filters and sequential Monte Carlo methods, we establish a first comprehensive and fully rigorous approach to bound systematic and random errors of cloning algorithms in continuous time. To this end we develop a method to compare different algorithms for particular classes of observables, based on the martingale characterization of stochastic processes. Our results apply to a large class of jump processes on compact state space, and do not involve any time discretization in contrast to previous approaches. This provides a robust and rigorous framework that can also be used to evaluate and improve the efficiency of algorithms.

\end{abstract}



\end{frontmatter}

% AOS,AOAS: If there are supplements please fill:
%\begin{supplement}[id=suppA]
%  \sname{Supplement A}
%  \stitle{Title}
%  \slink[doi]{10.1214/00-AOASXXXXSUPP}
%  \sdatatype{.pdf}" 
%  \sdescription{Some text}
%\end{supplement}

\section{Introduction}

Cloning algorithms have been introduced to the theoretical physics literature \cite{giardina, lecomte2007numerical} as numerical methods to study large deviations of particle currents and other dynamic observables in stochastic particle systems. They combine importance sampling with a stochastic selection mechanism which is used to evaluate numerically the scaled cumulant generating function for time-additive path functionals of stochastic processes. Based on classical ideas of evolutionary algorithms \cite{anderson1975random, grassberger2002go}, a fixed size population of copies of the original system evolves in parallel, subject to cloning or killing in such a way as to favor the realization of atypical trajectories contributing to rare events.
Various variants of the approach are now applied on a regular basis to different systems and large deviation phenomena of interest \cite{giardina2011simulating,hurtado2014thermodynamics,perez2019sampling}, including also current fluctuations of non-equilibrium lattice gas models \cite{hurtado2009test, hurtado2014thermodynamics,nemoto2018optimizing,chleboun2018current}, turbulent flows \cite{lestang2020numerical}, glassy dynamics \cite{pitard2011dynamic}, heat waves in climate models \cite{ragone2018computation} and pressure of the edge-triangle model \cite{giardina2021approximating}. 
Due to its widespread applications, the mathematical justification and convergence properties of the algorithm have recently become a subject of research interest with only partial progress. Formal approaches so far are based on a branching process interpretation of the algorithm in discrete time \cite{nemoto2017finite}, with limited and mostly numerical results in continuous time \cite{hidalgo2017finite, tchernookov2010list, nemoto2016population, brewer2018efficient}. \\

% the large deviations of dynamic observables of interacting lattice gases.  The cloning algorithm has been applied to a wide variety of physical systems, . However, there have been fewer studies on the mathematical justification of the algorithm: 

In this paper, we provide a novel interpretation of cloning algorithms through Feynman-Kac models and their particle approximations (see \cite{smc:theory:DM00, del2004feynman, del2013mean, del2003particle} for comprehensive reviews), which is itself an established approach to understanding sequential Monte Carlo methods and particle filtering.
Previous results provide rigorous control on convergence properties and error bounds of particle filters and related algorithms, mostly for models in discrete time, beginning with the chain of research initiated by \cite{smc:theory:Del96} with a recent survey provided in \cite{del2013mean}.
Fewer results address continuous-time dynamics, dating back to \cite{smc:theory:CL97} in the filtering context, with a Feynman-Kac-based treatment provided by \cite{smc:theory:DM00} and references therein; a survey of the filtering literature is provided by \cite[Chapter 9]{smc:theory:BC09}. In the current context, particularly relevant recent works include \cite{del2000moran, rousset2006control,del2012concentration,eberle2013quantitative, cerou2016central}. This literature generally considers diffusive dynamics and relies upon approximative time-discretisations of those dynamics. Adapting those results to the context of jump processes on locally compact state spaces, for which exact simulation from the dynamics is possible, we can establish first rigorous convergence results for the cloning algorithm in continuous time including $L^p$ bounds on the random error and bounds on the systematic error. These bounds include the explicit dependence on the clone size distribution, which is a key parameter of the cloning algorithm. The setting of finite activity pure jump processes in which cloning algorithms are primarily employed allows these algorithms to avoid time discretisation by simulating exactly from the law of the underlying process and allows the use of different approximating particle systems. Similar methods have been previously employed in the probabilistic rare event analysis literature in both discrete and continuous time, via explicit Feynman-Kac approximations, e.g. \cite{smc:methodology:CDLL06}, and splitting algorithms (see \cite{probability:rare:BGGLR16} and references therein); however, both the underlying processes and approximations considered are quite different to those for which cloning algorithms are usually employed. Practically, an important contribution of our approach is a systematic method to compare different cloning algorithms and particle approximations for particular classes of observables of interest, based on the martingale characterization of continuous-time stochastic processes. 
% This connection ensures the generalization of rigorous convergence results  to the setting of the cloning algorithm and in particular to establish bounds on the $L_p$ error of various estimators associated with the cloning algorithm. 

This framework provides a novel perspective on the underlying structure of cloning algorithms in terms of McKean representations \cite[Section 1.2.2]{del2013mean}, and can be used to systematically explore several degrees of freedom in the design of algorithms that can be used to improve performance, as illustrated in \cite{angeli2018rare} for current large deviations of the inclusion process \cite{chleboun2018current}. Here we focus on presenting full rigorous results obtained by applying this approach to a version of the classical cloning algorithm in continuous time \cite{lecomte2007numerical}. In contrast to previous work in the context of cloning algorithms \cite{nemoto2017finite, hidalgo2017finite}, our mathematical approach does not require a time discretization and works in the very general setting of a pure jump Markov process on a locally compact state space. This covers in particular any finite-state Markov chain or stochastic particle systems on finite lattices.
% with bounded total mass. %We comment on possible extensions of our results to locally compact state spaces in the discussion, but since in practice computer implementations of any system are necessarily finite, lifting this restriction is not essential for the practical applicability of our results.
\\

The paper is organized as follows. In Section \ref{section1} we introduce general Feynman-Kac models associated to pure jump Markov processes and show that they can be interpreted as the law of a non-linear Markov process, known as a {McKean interpretation} \cite{del2004feynman}. In Section \ref{section2} we introduce particle approximations for Feynman-Kac models, including classical mean-field versions and cloning algorithms. We provide generalized conditions for  convergence as our main result (proved in Section \ref{section_proofs}), and use this to establish rigorous convergence bounds for cloning algorithms.
In Section \ref{section3} we introduce large deviations and scaled cumulant generating functions (SCGF) of additive observables for pure jump Markov processes and discuss how the results presented  in Section \ref{section2} can be applied to estimate the SCGF. We conclude with a short discussion in Section \ref{discussion}.

%we introduce large deviations and scaled cumulant generating functions (SCGF) of additive observables for pure jump Markov processes to motivate our problem, and fit it into the general framework of Feynman-Kac semigroups. In Section \ref{section2} we introduce particle approximations including classical mean-field versions and cloning algorithms. We provide generalized conditions for  convergence as our first main result (proved in Section \ref{section_proofs}), and use this to establish rigorous convergence bounds cloning algorithms. In Section \ref{section3} we discuss how these results can be applied to our motivating question of large deviations of additive observables, and conclude with a short discussion in Section \ref{discussion}.

% irst provide the infinitesimal description of the classical mean field particle approximation as well as the cloning algorithm, and then we rigorously discuss their connection by comparing their generators and martingale characterizations. In Section \ref{section3} we present our main results on the convergence properties and sampling errors of the cloning algorithm for two practical estimators for the SCGF based on the rate of cloning events. We conclude with a short discussion.

\section{Mathematical Setting}\label{section1}

\subsection{Dynamics and Feynman-Kac models}

We consider a continuous-time homogeneous Feller process $\big(X_t:t\geq 0\big)$ taking values on a locally compact Polish state space $(E,\BB(E))$, where $\BB(E)$ is the Borel field on $E$. We denote by $\MM(E)$ and $\PP(E)$ the sets of measures and probability measures, respectively, on $(E,\BB(E))$. $\big( P(t):t\geq 0\big)$ describes the semigroup associated with $X_t$, which is considered as acting on the Banach space $\CC_b(E)$ of bounded continuous functions $f:E\to\R$, endowed with the supremum norm 
\begin{equation*}
    \|f\|=\sup_{x\in E}|f(x)|.
\end{equation*}
We use the standard notation $\P$ and $\E$ for the distribution and the corresponding expectation on the usual path space
\begin{equation*}
    \Omega:=\big\{ \o:[0,\infty)\to E \mathrm{\, right \, continuous \, with \, left \, limits}\big\}.
\end{equation*}
The measurable structure on $\Omega$ is given by the Borel $\sigma$-algebra induced by the \textit{Skorokhod topology} (see \cite{billingsley2013convergence}, Chapter 3). 
If we want to emphasize a particular initial condition $x\in E$ or distribution $\mu \in \PP(E)$ of the process we write $\P_x$ and $\E_x$, or $\P_\mu$ and $\E_\mu$, respectively. 
The semigroup $P(t)$ acts on bounded continuous functions $f$ and probability measures $\mu\in\PP(E)$  via
%\begin{equation*}
%    P(t)\, f(x)=\int_E P(t)\,(x,\,dy)f(y),\quad \mu P(t)\,(A)=\int_E P(t)\, (x,A).
%\end{equation*}
\begin{equation*}
    P(t)\, f(x)=\E_x \big[ f(X_t )\big]\ ,\quad \mu P(t)\,(f):=\int_E P(t) f(x) \mu (dx) =\E_\mu \big[ f(X_t )\big]\ ,
\end{equation*}
where the latter provides a weak characterization of the distribution $\mu P(t)$ at time $t\geq 0$.
Here and in the following we use the common notation $\mu (f)$ for expectations of $f\in \CC_b(E)$ w.r.t.\ measures $\mu$ on $E$.

Using the Hille-Yosida Theorem (see e.g.\ \cite{liggett2010continuous}, Chapter 3), it is possible to associate to the above Feller process an infinitesimal generator $\LL$ acting on a dense subset $\mathcal{D}\subset\CC_b(E)$ so that
\begin{equation*}
    \frac{d}{dt}P(t)\, f=\,\LL \big(P(t)\, f\big)\,=\,P(t)\, \LL (f),
\end{equation*}
for all $f\in\mathcal{D}$ and $t\geq 0$.

In this work, we restrict ourselves to nonexplosive pure jump Feller processes. We denote by $\l(x)$ the escape rate from state $x\in E$ and the target state is chosen with the probability kernel 
%To describe these processes, it is usual to introduce the \textit{escape rate function} $\l(x)$ such that $\l(x)dt+o(dt)$ is the probability that $X_t$ undergoes a jump during $[t,t+dt]$ starting from the state $X_t=x$. When a jump occurs, $X_{t+dt}$ is then distributed with the probability kernel 
$p(x,dy)$, so that the overall transition rate is
\begin{equation}\label{rates}
W(x,dy):=\l(x)\cdot p(x,dy)
\end{equation}
for $(x,y)\in E^2$. 
We assume $\l:E\to [0,\infty)$ to be a strictly positive, bounded and continuous function 
% such that $\inf_x \lambda(x)>0$ 
and $x\mapsto p(x,A)$ to be a continuous function for every $A\in \BB(E)$. Under these assumptions, the pure jump process possesses an infinitesimal generator \cite[p. 162]{ethier2009markov} with full domain $\mathcal{D}=\CC_b(E)$ given by
\begin{equation*}
\LL (f)(x)=\int_E W(x,dy)[f(y)-f(x)], \quad \forall f\in \CC_b(E), \, x\in E. 
\end{equation*}

Along with jump processes on continuous spaces such as continuous-time random walks on $\R^d$ (see e.g. \cite{ctrw}), this setting includes in particular any finite-state continuous-time Markov chain. 
Typical compact examples we have in mind are given by stochastic particle systems on $E=S^\L$, with finite local state space $S$ and lattice $\L$ which can be finite or countably infinite.
%is finite or countably infinite, so that $E$ is compact. 
These include spin systems with $S=\{ -1,1\}$ or exclusion processes with $S=\{0,\,1\}$, in which particles can jump only onto empty sites. Stochastic particle systems such as zero-range processes with $S=\N_0$ are locally compact as long as the lattice $\L$ is finite (see e.g.\ \cite{liggett2} for details).

We will study Feynman-Kac models associated to the jump process by tilting its generator with a diagonal part or potential, which arise in many applications including dynamic large deviations, as explained in detail in Section \ref{section3}.

\begin{lemma}\label{tilted}
Consider a potential function $\VV\in C_b (E)$ and the tilted generator
\begin{equation}\label{fkgen}
    \LL^\VV (f)(x):=\LL (f)(x) +\VV (x) f(x)\quad\mbox{defined for all }f\in C_b (E)\ .
\end{equation}
Then the family of operators $\big( P^\VV (t):t\geq 0\big)$ with $P^\VV :C_b (E)\to C_b (E)$, defined as the solution to the backward equation
\begin{equation}\label{nusemi}
    \frac{d}{dt} P^\VV (t)f=\LL^\VV \big( P^\VV (t)f\big)\quad\mbox{with}\quad P^\VV (0)f=f
\end{equation}
for all $f\in C_b (E)$, forms a non-conservative semigroup, the so-called Feynman-Kac semigroup, and $\LL^\VV$ is its infinitesimal generator in the sense of the Hille-Yosida Theorem.
% For pure jump processes, for any $k \in \mathbb{R}$ the family of operators $\big(P_k(t):t\geq 0\big)$ on $\CC_b(E)$ defined by
% \begin{equation}\label{P_k}
%     P_{k}(t)f(x)\,:=\, \E_{x}\big[ f\big(X_t\big)\,e^{kt A_t}\big],
% \end{equation}
% with $f\in \CC_b(E)$, is well defined and it is a non-conservative semigroup, the so-called tilted semigroup. 
% Moreover, the infinitesimal generator associated with $\big( P_{k}(t):t\geq 0\big)$, in the sense of the Hille-Yosida Theorem, can be written in the form
% \begin{equation}\label{tilt_gen}
% \LL_k (f)(x)=\int_E W(x,dy)[e^{kg(x,y)}f(y)-f(x)]\,+\,kh(x)f(x),
% \end{equation}
% for $f\in \CC_b(E)$ and all $x\in E$, with $g$ and $h$ the bounded continuous functions which characterize $A_T$ via \eqref{oss_jump}. In particular, the semigroup $P_{k}(t)$ satisfies the differential equations
% \begin{equation}\label{ddt}
%     \frac{d}{dt}P_{k}(t)f\,=\, P_{k}(t)\LL_k (f)\,=\,\LL_k \big(P_k(t)f\big),
% \end{equation}
% for all $f\in\CC_b(E)$ and $t\geq 0$.
\end{lemma}
\begin{proof}
See \cite{liggett2010continuous}, Theorem 3.47.

\end{proof}

In order to control the asymptotic behaviour of $P^\VV (t)$, we make the following assumption, which closely resembles \cite[Assumption 1]{rousset2006control}, on asymptotic stability.
% , which implies \eqref{asymptotic_stability}.

\begin{assumption}[Asymptotic Stability]\label{ass_expstab_generic}
The spectrum of $\LL^\VV =\LL+\VV$ \eqref{fkgen} is bounded by a principal eigenvalue $\lambda_0$. Moreover, $\lambda_0$ is associated to a positive eigenfunction $r\in\CC_b(E)$ and an eigenmeasure $\mu_\infty\in\PP(E)$. Finally, there exist constants $\a>0$ and $\rho\in (0,1)$ such that
\begin{equation}\label{asssta}
    \big\| e^{-t\lambda_0} P^\VV (t)f(\cdot)-\mu_\infty(f)\big\|\,\leq\, \|f\|\cdot \a\rho^t\ ,
\end{equation}
for every $t\geq 0$ and $f\in\CC_b(E)$.
\end{assumption}

 Asymptotic stability is for example guaranteed for all irreducible, finite-state continuous-time Markov chains which necessarily have a spectral gap. For alternative sufficient conditions implying asymptotic stability in a more general context including continuous state spaces, see Appendix \ref{appendix_stability}.
 
 We introduce the measures $\nu_{t,\mu_0}$ for any general initial distribution $\mu_0\in\PP(E)$ and $t\geq 0$, defined by
\begin{equation}\label{nu}
    \nu_{t,\mu_0}(f):=\mu_0\big( P^\VV (t)f\big),
\end{equation}
for any $f\in\CC_b(E)$. In the literature \cite{del2004feynman}, $\nu_{t}$ is known as the \textit{unnormalised t-marginal Feynman-Kac measure}. Applying Lemma \ref{tilted}, we can see that $\nu_{t}$ solves the evolution equation
\begin{equation}\label{nu_eq}
    \frac{d}{dt}\nu_{t,\mu_0}(f)=\nu_{t,\mu_0}\big(\LL^\VV (f)\big)=\nu_{t,\mu_0}\big(\LL (f)+\VV\cdot f \big),
\end{equation}
for any $f\in\CC_b(E)$, $t\geq 0$ and $\mu_0\in\PP(E)$. The measures with which one can most naturally associate a process are the corresponding \textit{normalised t-marginal Feynman-Kac measures} in $\PP(E)$,
\begin{equation}\label{def_mu}
    \mu_{t,\mu_0} (f):=\frac{\nu_{t,\mu_0}(f)}{\nu_{t,\mu_0}(1)},
\end{equation}
defined for any $t\geq 0$ and $f\in\CC_b(E)$.

% For simplicity, in the rest of this article the initial distribution $\mu_0$ is fixed and we write $\mu_t$ (resp. $\nu_t$) instead of $\mu_{t,\mu_0}$ (resp. $\nu_{t,\mu_0}$).

% The results presented in this and the next section hold for general t-marginal Feynman-Kac measures associated to a (homogeneous) pure jump generator $\LL$ with bounded rates and a generic potential $\VV\in \CC_b(E)$ via
% \begin{equation}\label{nu_eq_general}
%     \frac{d}{dt}\nu_t (f)\,=\,\nu_t\big(\LL (f)\,+\,\VV \cdot f\big)\ .
% \end{equation}
% We use this more general notation for simplicity, and write $P^\VV (t)$ for the associated tilted semigroup \eqref{P_k} with $\nu_t (f)=\mu_0 \big( P^\VV (t)f\big)$ for any $f\in\CC (E)$ and

Observe that, as a direct consequence of asymptotic stability (Assumption \ref{ass_expstab_generic}), there exist constants $\tilde{\a}\geq 0$ and $0<\rho<1$ such that for any $f\in\CC_b(E)$,
\begin{equation}\label{lemma_expstab}
    \big|\mu_{t,\mu_0}(f)-\mu_\infty(f)\big|\leq \|f\|\cdot {\tilde{\a}\rho^{t}} \ ,
\end{equation}
for any $t\geq 0$ and initial distribution $\mu_0\in\PP(E)$. In particular $\mu_{t,\mu_0}$ converges weakly to $\mu_\infty$, as $t\to\infty$. 
Indeed, by definition of $\mu_{t,\mu_0}$ \eqref{def_mu} and then by asymptotic stability (Assumption \ref{ass_expstab_generic}),
\begin{equation}\label{assstab}
    \frac{\mu_\infty(f)\, -\, \|f\|\,\alpha\cdot \rho^t}{1+\alpha\cdot \rho^t}\leq\mu_{t,\mu_0}(f)\,=\, \frac{\mu_0\big(e^{-t\lambda_0}P^\VV (t)f\big)}{\mu_0\big(e^{-t\lambda_0}P^\VV (t)1\big)}\,\leq\, \frac{\mu_\infty(f)\,+\, \|f\|\,\alpha\cdot \rho^t}{1-\alpha\cdot \rho^t} \ ,
\end{equation}
for any $t> -\log\alpha /\log \rho$ and for some constant 
% $\tilde{\a}>0$
$\alpha >0$. This gives the bound \eqref{lemma_expstab} for any $t$ large enough. Increasing $\tilde{\a}$ accordingly to ensure that the bound holds also for small $t$, we obtain \eqref{lemma_expstab} for any $t\geq 0$.

For simplicity, in the rest of this article the initial distribution $\mu_0$ is fixed and we write $\mu_t$ (resp. $\nu_t$) instead of $\mu_{t,\mu_0}$ (resp. $\nu_{t,\mu_0}$).

\subsection{McKean Interpretations}

%Proposition \ref{Lk_mu} 
% Lemma \ref{lemma_conv_Lk} states in particular that $\L_k$ can be well approximated as a time average of $\mu_{t}(\VV_k)$. 
Now, we want to outline the evolution of the time-marginal distribution $\mu_{t}$ in terms of interacting jump-type infinitesimal generators. The content presented in the rest of this section is based on the works of Del Moral and Miclo \cite{del2004feynman, del2013mean, smc:theory:DM00}.
% since, as we have seen, the large-deviation conditioning problem can be reformulated through the Feynman-Kac theory. 
In this established framework it is possible to define generic Markov processes with time marginals $\mu_t$ and then use Monte Carlo sampling techniques to approximate those marginals.
% , and thus the SCGF $\L_k$.

\begin{lemma}\label{dmu}
For every $f\in\CC_b(E)$ and $t\geq 0$, the normalised $t$-marginal $\mu_t$ \eqref{def_mu} solves the non-linear evolution equation
\begin{equation}\label{mu_evolution}
\frac{d}{dt}\mu_t(f)=\mu_t\big(\LL (f)\big)+\mu_t(\VV  f)-\mu_{t}(f)\cdot\mu_t(\VV ).
\end{equation}
\end{lemma}

\begin{proof}
Using the evolution equation \eqref{nu_eq} of $\nu_t$, we see that
\begin{align*}
\frac{d}{dt}\mu_t(f)&=\frac{d}{dt}\frac{\nu_t(f)}{\nu_t(1)}\\
&= \frac{1}{\nu_t(1)}\cdot \nu_t\big(\LL (f)\,+\, \VV\cdot f \big)-\frac{\nu_t(f)}{\nu_t(1)^2}\,\nu_t\big(\LL (1)\,+\, \VV\big)\\
&= \mu_t\big(\LL (f)\big)+\mu_t(\VV  f)-\mu_{t}(f)\cdot\mu_t(\VV )\ .
\end{align*}

\end{proof}

The evolution equation \eqref{mu_evolution} results from the unique decomposition of the non-conservative generator $\LL+\VV$ into a conservative and a diagonal part given by the potential $\VV$. The latter, together with the normalization of $\nu_t$, leads to the nonlinear second part in \eqref{mu_evolution} which we want to rewrite to be in the form of another infinitesimal generator, that we denote by $\widetilde{\LL}_{\mu_t}$. Since \eqref{mu_evolution} is non-linear in $\mu_t$, this depends itself on the current distribution such that
\begin{equation}\label{general_mcKean}
    \mu\big(\widetilde{\LL}_\mu (f)\big)\,=\, \mu(\VV f)\,-\, \mu(f)\cdot \mu(\VV)\ ,
\end{equation}
for every $\mu\in \PP(E)$ and $f\in\CC_b(E)$. 
The choice of the non-linear generator $\widetilde{\LL}_\mu$ is not unique, leading to various representations of the form
\begin{equation}\label{widetildeLL}
    \widetilde{\LL}_\mu (f)(x)=\int_E \widetilde{W}(x,y)\big(f(y)-f(x)\big)\mu(dy)\ ,
\end{equation}
where $ \widetilde{W}(x,y)\mu(dy)$ is the overall transition kernel of $\widetilde{\LL}_\mu$ and depends on the current distribution $\mu$. 

\begin{lemma}[Sufficient conditions]\label{sufficient_conditions}
An infinitesimal generator in the form \eqref{widetildeLL} satisfies condition \eqref{general_mcKean} if and only if 
\begin{equation*}
    \mu\big( \widetilde{W}(\cdot,x)-\widetilde{W}(x,\cdot)\big)\,=\, \VV(x)-\mu(\VV)\ ,
\end{equation*}
for all $\mu\in\PP(E)$ and $x\in E$. In particular, a sufficient condition on $\widetilde{\LL}_\mu$ \eqref{widetildeLL} for \eqref{general_mcKean} to hold is
\begin{equation*}
    \widetilde{W}(y,x)\,-\,\widetilde{W}(x,y)\,=\, \VV(x)\,-\,\VV(y)\ ,
\end{equation*}
for all $x,y\in E$.
\end{lemma}

\begin{proof}
It is enough to observe that
\begin{align*}
    \mu\big(\widetilde{\LL}_\mu (f)\big)&=\int_{E^2} \widetilde{W}(x,y)\big(f(y)-f(x)\big)\mu(dy)\mu(dx)\\
    &=\int_{E^2} \big(\widetilde{W}(y,x)-\widetilde{W}(x,y)\big)\,f(x)\mu(dy)\mu(dx)\ .
\end{align*}

\end{proof}

Combining $\widetilde{\LL}_\mu$ with the linear part $\LL$ of \eqref{mu_evolution} into a so-called \textit{McKean generator} on $\CC_b(E)$, 
\begin{equation}\label{mcKean_generator}
    \overline{\LL}_\mu\,:=\, \LL \,+\, \widetilde{\LL}_\mu\quad\mbox{for all }\mu\in\PP (E)\ ,
\end{equation}
the evolution equation \eqref{mu_evolution} can be written as
\begin{equation*}
    \frac{d}{dt} \mu_t(f)\,=\, \mu_t\big(\overline{\LL}_{\mu_t}(f)\big)\ ,
\end{equation*}
for every $f\in \CC_b(E)$ and $t\geq 0$. Therefore, the normalized Feynman-Kac marginal $\mu_t$ can be interpreted as the law of a Markov process $\big(\overline{X}_t\,:\, t\geq 0\big)$ on $E$, associated to the family of generators $\big(\overline{\LL}_{\mu_t} :t\geq 0\big)$. This process is also known as a \textit{McKean representation} of the process associated to the Feynman-Kac measure $\mu_t$, and it is non-linear and in particular time-inhomogeneous. This can be formulated using the propagator 
\begin{equation}\label{normalized_prop}
    \Theta_{t,T}f(x):=\frac{P^{\VV}(T-t)f(x)}{\mu_t\big(P^{\VV}(T-t)1\big)}\quad\mbox{such that}\quad \mu_T (f)=\mu_t (\Theta_{t,T} f)
\end{equation}
for all $0\leq t\leq T$, which follows directly from the definition of $\mu_t$ \eqref{def_mu} and the semigroup characterizing the time evolution for $\nu_t$ \eqref{nusemi}.

While the time evolution of $\mu_t$ is uniquely determined by \eqref{mu_evolution} and therefore independent of the choice of \eqref{mcKean_generator}, Lemma \ref{sufficient_conditions} leads to various different McKean representations of the form \eqref{widetildeLL} (see e.g.\ \cite{angeli2018rare, rousset2006control}), that can be characterized by the operator $\widetilde{W}$. One common choice related to algorithms in  \cite{giardina, lecomte2007numerical} is
\begin{equation}\label{mcKean1}
    \widetilde{W}_c(x,y)\,=\, \big(\VV(x)-c\big)^- \,+\, \big(\VV(y)-c\big)^+\ ,
\end{equation}
where $c\in \R$ is an arbitrary constant, and we use the standard notation $a^+ =\max\{0,a\}$ and $a^- = \max\{0,-a\}$ for positive and negative part of $a\in\R$.

One other possible representation of \eqref{widetildeLL} we want to mention explicitly here is given by
\begin{equation}\label{mcKean2}
    \widetilde{W}(x,y)\,=\, \big(\VV (y)-\VV (x)\big)^+\ .
\end{equation}
This corresponds to a pure jump process on $E$ in which every jump strictly increases the value of the potential $\VV $ in contrast to the previous representation \eqref{mcKean1}. We will see in the next section that $\VV $ can be interpreted as a fitness potential for the overall process. Further McKean representations of \eqref{mu_evolution} are discussed in \cite{angeli2018rare}, here we focus on cloning algorithms which are based on \eqref{mcKean1}.
%, which is invariant under addition of constants represented by the parameter $c\in\R$ in the representation \eqref{evolution_mu}. 

%fitness of the process and the rates depend on both departure state and target state, which is in general computationally more expensive to implement that rates in \eqref{Lmu} and \eqref{Lmu'}, but can still be feasible due to semplifications in many concrete examples \cite{angeli2018rare}.

\section{Interacting Particle Approximations}\label{section2}

Independent of the particular representation, the rates of the McKean process $(\overline{X}_t :t\geq 0)$ depend on the distribution $\mu_t$ itself, which is in general not known. A standard approach is to sample such processes through particle approximations \cite{del2003particle}, which involve running, in parallel, $N$ copies or clones $\xi_t:=(\xi_t^1,\dots,\xi_t^N)\in E^N$ of the process (called particles), and then approximating $\mu_t$ by the empirical distribution $m(\xi_t) $ of the realizations. For any $\underline{x}\in E^N$ the latter is defined as
\begin{equation}\label{empirical_distrib}
    m(\underline{x})(dy):=\frac{1}{N}\sum_{i=1}^N \delta_{x_i}(dy)\;\in\,\PP(E).
\end{equation}

% Throughout the section, we assume that the initial condition of the particle approximation is chosen as
% \begin{equation}\label{inicon}
%     \xi_0^1 ,\ldots ,\xi_0^N\quad\mbox{are i.i.d.r.v's with distribution}\ \mu_0\ .
% \end{equation}

We write $\overline{L}^N$ for the infinitesimal generator of an $N$-particle system $\xi_t$ and also call this an IPS generator, and denote the associated empirical distribution as
\begin{equation}\label{emea}
    \mu^N_t(\cdot):=m(\xi_t)(\cdot).
\end{equation}
We denote by
\begin{equation*}
    \Gamma_{\overline{L}^N}(\gamma,\varphi):=\overline{L}^N(\gamma\cdot \varphi)-\gamma\cdot \overline{L}^N(\varphi)-\varphi\cdot \overline{L}^N(\gamma)\ ,\qquad \gamma,\varphi\in \CC_b(E^N)\ ,
\end{equation*}
the standard \textit{carr\'{e}-du-champ} operator associated to the generator $\overline{L}^N$.

\subsection{A general convergence result}

The full dynamics can be set up in various different ways such that $\mu_t^N\to \mu_t$ converges in an appropriate sense as $N\to\infty$ for any $t\geq 0$. Theoretical convergence results can be obtained under the following assumptions, which are fulfilled by standard mean field particle approximations (as shown in Section \ref{subsection_MF}) and cloning algorithms (Section \ref{subsection_cloning}).

%\subsection{A generalized convergence result}

\begin{assumption}\label{ass_IPS}
\begin{subequations}
Given a family of McKean generators $\big(\overline{\LL}_\mu\big)_{\mu\in\PP(E)}$ \eqref{mcKean_generator} on $\CC_b(E)$, we assume that the sequence of particle approximations $(\xi_t :t\geq 0)$ with generators $(\overline{L}^N)_{N\in\N}$ on $\CC_b(E^N)$ satisfies \begin{align}
    \overline{L}^N(F)(\underline{x})\,&=\, m(\underline{x})\big(\overline{\LL}_{m(\cdot)}(f)\big)\ ,\label{ass_generator}\\
    \Gamma_{\overline{L}^N}( F,\,F)(\underline{x})\,&=\,\frac{1}{N}\,m(\underline{x})\big(G_{m(\cdot)}(f,f)\big)\,+\,\Ocarre{N}{\underline{x}}{f}\ ,\label{ass_carre}
\end{align}
for mean-field observables $F\in\CC_b(E^N)$ of the form $F(\underline{x})=m(\underline{x})(f)$, $f\in \CC_b(E)$.
%Here $O\big(\frac{1}{N^2}\big)$ denotes a function of $\underline{x}$, depending also on $\|f\|$, that vanishes uniformly in $\underline{x}$ as $N\to\infty$, with rate of convergence $1/N^2$ or faster. 
Here $\Ocarre{N}{\underline{x}}{f}$ is a function of $\underline{x},$ and $N$, such that there exists a constant $C>0$ (independent of $N$, $f$) with
$$\|\Ocarre{N}{\,\cdot\,}{f}\|\leq C\,\frac{\|f\|^2}{N^2}\ ,$$
% $$\sup_{\underline{x}\in E^N}|\Ocarre{N}{\underline{x}}{f}|\leq C\,\frac{\|f\|^2}{N^2}\ ,$$
for any $f\in\CC_b(E)$ and $N\in\N$.
$\big(G_{\mu}\big)_{\mu\in\PP(E)}$ is a family of bilinear operators $G_{\mu}:\CC_b(E)\times \CC_b(E)\to \CC_b(E)$ independent of the population size $N$, such that
\begin{equation*}
    \sup_{\mu\in\PP(E)}\sup_{\|f\|\leq 1}\| G_\mu(f,f)\|< \infty\ .
\end{equation*}
Furthermore, we assume there exists a constant $K<\infty$ (independent of $N$), such that for all $N\in\N$, almost surely,
\begin{equation}\label{bounded_jumps}
    \sup_{t\geq 0}\left \vert \left\{ i \in 1,\ldots,N : \xi^i_t\neq \xi^i_{t-}\right\}\right\vert \,\leq\, K\ .
 \end{equation}
% for almost every realization $\xi_t$, and every $t>0$ and $N\in\N$.\\
For the initial condition of the particle approximation we assume that
\begin{equation}\label{inicon}
    \xi_0^1 ,\ldots ,\xi_0^N\quad\mbox{are i.i.d.r.v's with distribution}\ \mu_0\ .
\end{equation}

\end{subequations}
\end{assumption}

\begin{remark} Test functions of the form
 \[
 F(\underline{x})=m(\underline{x})(f) =\frac{1}{N}\sum_{i=1}^N f(x_i )
 \]
 describe mean-field observables averaged over the particle ensemble which are generally of most interest, e.g.\ for the estimator \eqref{LTN} of the SCGF it is sufficient to consider such functions, as shown in Section \ref{4.1}. In general the goal is to approximate $\mu_t (f)$ for a given $f\in\CC (E)$, so it is natural to set up the auxiliary particle approximation in a permutation invariant way and use mean-field observables.\\

To better understand the above assumptions, recall that the carr\'{e} du champ of an interacting particle system is a quadratic operator associated to the fluctuations of the process, whereas the generator determines the expected behaviour of the observables $F({\xi}_t)$. Thus, Assumption \ref{ass_IPS} implies that trajectories of mean-field observables in a particle approximation coincide in expectation with average trajectories of the McKean representation they are based on \eqref{ass_generator}, and concentrate on their expectation with diverging $N$ \eqref{ass_carre}. We include the operators $G_\mu$ explicitly in  \eqref{ass_carre}, because it allows the condition to be stated in a convenient form and we anticipate it being useful in further analysis. Condition \eqref{bounded_jumps} assures that at any given time only a bounded number of particles can change their state, which is a mild technical assumption, 
%since we are dealing with continuous-time jump processes.
necessary to allow the application of Lemma \ref{lemma6.2} in the proof of the $L^p$ error estimates. 
\end{remark}

%Under Assumption \ref{ass_IPS}, and provided that the initial condition of the particle approximation is chosen as
%\begin{equation}\label{inicon}
%    \xi_0^1 ,\ldots ,\xi_0^N\quad\mbox{are i.i.d.r.v's with distribution}\ \mu_0\ ,
%\end{equation}
%we obtain the following convergence results.

\begin{thm}\label{thm_weak}
Consider a sequence of particle approximations satisfying Assumption \ref{ass_IPS} with empirical distributions $\mu_t^N$ \eqref{emea}.
% Let $(\overline{L}^N)_{N\in\N}$ be a sequence of IPS generators satisfying Assumption \ref{ass_IPS}.
%, with initial condition $\xi_0$ according to \eqref{inicon}
Under Assumption \ref{ass_expstab_generic}, for every $p\geq 2$ there exists a constant $c_p>0$ independent of $N$ and $T$ such that
\begin{equation}\label{Lp_estimate}
    \sup_{T\geq 0}\,\mathbb { E } \left[ \left( \mu _ { T} ^ { N } ( f ) - \mu _ { T } ( f ) \right) ^ { p } \right] ^ { 1 / p } \,\leq\, \frac{c_p\|f\|}{N^{1/2}}  \ ,
\end{equation}
for any $f\in\CC_b(E)$. Furthermore, there exists a constant $c^\prime >0$ independent of $N$ and $T$ such that
\begin{equation}\label{bias_estimate}
    \sup_{T\geq 0}\,\left| \mathbb { E } \left[ \mu _ { T } ^ { N } ( f ) \right] - \mu _ { T } ( f ) \right| \,\leq\, \frac { c ^ { \prime }\| f \| } { N }\ ,
\end{equation}
for any $f\in\CC_b(E)$ and $N\in\N$ large enough.

\end{thm}

\begin{remark}
The constants $c_p$ and $c'$ depend on the Feynman-Kac model of interest, on the choice of the McKean model and on the considered interacting particle approximation. 
\end{remark}

The proof, presented in Section \ref{section_proofs}, is an adaptation of the results in \cite{rousset2006control} and makes use of the propagator \eqref{normalized_prop} of $\mu_t$ and the martingale characterization of $(\xi_t :t\geq 0)$.

\begin{remark}
Observe that, by Markov's inequality, Theorem \ref{thm_weak} implies 
\begin{equation*}
    \P\Big(\big|\mu_t^N(f)-\mu_t(f)\big|\geq \e\Big)\,\leq\, \frac{c_{p}\cdot \|f\|^p}{\e^p\cdot N^{p/2}}\ ,
\end{equation*}
for every $\e,\,t>0$, $f\in\CC_b(E)$, $N\geq K$ and $p\geq 2$, where $c_{p}>0$ does not depend on $N$. In particular, considering $p>2$, we can see that
\begin{equation}\label{as_conv}
    \mu_t^N(f)\,\rightarrow\,\mu_t(f)\quad \mathrm{a.s.}
\end{equation}
as $N\to\infty$, for any $f\in \CC_b(E)$, by a Borel-Cantelli argument. The existence of a countable determining class allows this to be further strengthened to the almost sure convergence of $\mu_t^N$ to $\mu_t$ in the weak topology (see, for example, \cite[Theorem 4]{mcmc:theory:SDDP18}).
\end{remark}

It is important to clarify that the estimators of the Feynman-Kac distribution $\mu_t$ given by the empirical measures $\mu^N_t$ usually have a bias, i.e. $\E[\mu^N_t(f)]\neq \mu_t(f)$ for $f\in\CC_b(E)$, which vanishes only asymptotically, as illustrated in Theorem \ref{thm_weak}. This arises from the non-linear time evolution of $\mu_t$. However, it is straightforward to derive unbiased estimators of the unnormalized measures $\nu_t$ \eqref{nu}, as shown by the following result.

\begin{prop}[Unbiased Estimators]\label{prop_unbiasedness}
Consider a sequence of particle approximations satisfying \eqref{ass_generator} and initial condition \eqref{inicon}, with empirical distributions $\mu_t^N$ \eqref{emea}. Then, the \textit{unnormalized empirical measure} 
\begin{equation*}
    \nu_t^N (f):=\nu_t^N (1)\mu_t^N (f)\quad\mbox{with}\quad\nu_t^N (1) :=\exp \Big( \int_0^t \mu_s^N  (\VV  )ds\Big)\ ,
\end{equation*} 
is an unbiased estimator of the unnormalized $t$-marginal $\nu_t$ \eqref{nu}, i.e.
\begin{equation}\label{gronwall1}
\E \big[ \nu_t^N (f)\big] =\nu_t (f)\quad\mbox{for all }t\geq 0 \mbox{ and }N\geq 1\ ,
\end{equation}
for any $f\in\CC_b(E)$.
\end{prop}
\begin{proof}
First observe that $\E\big[\nu_0^N(f)\big]=\nu_0(f)$. Indeed, $\nu^N_0(f)=\mu^N_0(f)$ is the average of $N$ i.i.d. random variables with law $f_\#\mu_0$, and $\mu_0$ corresponds to the initial distribution of $\nu_t=\nu_{t,\mu_0}$ \eqref{nu}.

Note that $\E\big[\nu_t^N(f)\big]$ satisfies the evolution equation
\begin{equation}\label{vte}
\frac{d}{dt}\E \big[ \nu_t^N (f)\big] =\E\Big[ \nu_t^N (f)\mu_t^N (\VV  )+\nu_t^N (1)\overline{L}^N \mu_t^N (f)\Big]\ .
\end{equation}
Moreover, by assumption \eqref{ass_generator} and using the characterization of $\overline{\LL}_\mu$ \eqref{general_mcKean}-\eqref{mcKean_generator}, we have
\[
\overline{L}^N\mu_t^N (f)=\mu_t^N( \LL  f)+\mu_t^N(\VV \, f)-\mu_t^N(\VV )\cdot\mu_t^N(f)\ .
\]

Inserting into \eqref{vte}, this simplifies to
\[
\frac{d}{dt}\E \big[ \nu_t^N (f)\big] =\E\big[ \nu_t^N (\LL  f)+\nu_t^N (\VV  f)\big]\ .
\]
Since $\LL  +\VV $ also generates the time evolution of $\nu_t (f)$ \eqref{nu_eq}, a simple Gronwall argument with $\E \big[ \nu_0^N (f)\big] =\nu_0 (f)$ gives \eqref{gronwall1}.

\end{proof}

A generic version of interacting particle systems, directly related to the above McKean representations has been studied in the applied probability literature in great detail \cite{del2003particle,rousset2006control}, providing quantitative control on error bounds for convergence. After reviewing those results in the next subsection, we present a different approach taken in the theoretical physics literature under the name of cloning algorithms \cite{giardina,giardina2011simulating}, which provides some computational advantages but lacks general rigorous error control so far \cite{nemoto2017finite, hidalgo2017finite}. %Our main results are Theorem \ref{L&carre} and Proposition \ref{thm_Lp}, which underline the common aspects of these two approaches and provide first rigorous error bounds for cloning algorithms in continuous time.

\subsection{Mean Field Particle Approximation}\label{subsection_MF}

The most basic particle approximation is simply to run the McKean dynamics in parallel on each of the particles, replacing the distribution $\mu_t$ by the empirical measure. Formally, the mean field particle model $(\xi_t :t\geq 0)$ with $\xi_t=(\xi^i_t :i=1,\dots,N)$ associated to a McKean generator $\overline{\LL}_{\mu_t}$ \eqref{mcKean_generator}, is a Markov process on $E^N$ with homogeneous infinitesimal generator $\overline{L}^N$ defined by
\begin{equation}\label{overlineLLN}
   \overline{L}^N (F)(x^1,\dots,x^N):=\sum_{i=1}^N \overline{\LL}_{m(\underline{x} )}^{(i)}(F)(x^1,\dots,x^i,\dots,x^N),
\end{equation}
for any $F\in\CC_b(E^N)$. Here $\overline{\LL}_{m(\underline{x})}^{(i)}$ denotes the McKean generator $\overline{\LL}_{m(\underline{x})}$ \eqref{mcKean_generator} acting on the function $x^i\mapsto F(x^1,\dots,x^i,\dots,x^N)$, where the dependence on $\mu$ has been replaced by the empirical distribution $m(\underline{x})$.

In analogy to the decomposition $\overline{\LL}_\mu =\LL +\widetilde{L}_\mu$ in \eqref{mcKean_generator}, the generator \eqref{overlineLLN} can be decomposed as $\overline{L}^N =L^N +\widetilde{L}^{N}$ with
\begin{align}
L^N (F)(\underline{x})&:=\sum_{i=1}^N \LL^{(i)}  (F)(\underline{x})\ ,\label{Lmut}\\
\widetilde{L}^N (F)(\underline{x})&:=\sum_{i=1}^N \widetilde{\LL}_{m(\underline{x})}^{(i)}  (F)(\underline{x})\ ,\label{Lsel}
\end{align}
where $\LL^{(i)}$ and $\widetilde{\LL}_{m(\underline{x})}^{(i)}$ stand respectively for the operators $\LL $ and $\widetilde{\LL}_{m(\underline{x})}$ acting on the function $x^i\mapsto F(\underline{x})$, i.e. only on particle $i$.
%, with $\underline{x}:=(x_1,\dots,x_N)\in E^N$.

%In this context, the cloning and killing parts lead to mean-field interactions between the particles since the jumping rates depend on the empirical distribution of the particles.
Moreover, using representation \eqref{widetildeLL} for $\widetilde{\LL}_\mu$, we can write
\begin{equation}\label{Lkill_i}
    \widetilde{\LL}_{m(\underline{x})}^{(i)} (F)(\underline{x})\,=\, \frac{1}{N}\,\sum_{j=1}^N \widetilde{W}(x_i,\,x_j)\,\big( F(\underline{x}^{i,x_j})-F(\underline{x})\big)  \ ,
\end{equation}
with $\underline{x}^{i,y}:=(x_1,\dots,x_{i-1},y,x_{i+1},\dots,x_N)$, which introduces an interaction between the particles. In this decomposition, \eqref{Lmut} generates the so-called
\textit{mutation dynamics}, where the particles evolve independently under the dynamics given by the infinitesimal generator $\LL$ of the original process, whereas \eqref{Lsel} generates the \textit{selection dynamics}, which leads to mean-field interactions between particles. With \eqref{Lkill_i} the state of particle $i$ gets replaced by that of particle $j$ with rate $\frac{1}{N}\widetilde{W}(x_i,\,x_j)$. The total selection rate in the particle approximation is $\frac{1}{N}\sum_{i,j=1}^N \widetilde{W}(x_i,\,x_j)$, and depends on the McKean representation, in particular the choice of $\widetilde{\LL}_\mu$ in \eqref{widetildeLL}.

From general practical experience it is favourable to minimize the total selection rate in order to improve the estimator's asymptotic variance; it's widely understood in the SMC literature that eliminating unnecessary selection events can significantly improve estimator variances, see, for example, \cite[Section 7.2.1, 7.4.2]{del2004feynman} and \cite{Gerber2019}. For mean-field particle approximations this suggests that \eqref{mcKean2} is preferable to \eqref{mcKean1} since
\[
\widetilde{W} (x,y)=\big( \VV(y)-\VV (x)\big)^+ \leq \widetilde{W}_c (x,y)=\big( \VV(x)-c\big)^- +\big( \VV(y)-c\big)^+
\]
for all $x,y\in E$ and $c\in\R$. In view of Lemma \ref{sufficient_conditions}, minimizing the total selection rate pertains to maximizing $\sum_{i=1}^N \VV(x_i)$, and $\VV$ can be interpreted as a fitness function. With \eqref{mcKean2} every selection event therefore increases the fitness of the particle ensemble, which is not necessarily the case with \eqref{mcKean1}, and there are even more optimal choices than \eqref{mcKean2} in that sense as discussed in \cite{angeli2018rare}\footnote{As a side remark, the mutation part of the McKean dynamics (which is fixed for mean-field particle particle approximations by \eqref{Lmut}), can naturally also decrease the fitness of the ensemble.}. On the other hand, depending on the particular application, implementing particle approximations with lower total selection rate could be computationally more expensive, leading to a trade-off in lower values for $N$ to be accessible in practice. This is discussed in \cite{angeli2018rare} for a particular example, and is not the subject of this paper.

In order to motivate the choice of the cloning algorithm in the next subsection which is based on the selection rates \eqref{mcKean1}, we note that one can write \eqref{Lsel} as
\begin{align}
\widetilde{L}^N (F)(\underline{x})=&\sum_{i=1}^N \big( \VV (x_i)-c\big)^- \frac{1}{N} \sum_{j=1}^N \big( F(\underline{x}^{i,x_j})-F(\underline{x})\big) \nonumber\\
&+\sum_{i=1}^N \big( \VV (x_i)-c\big)^+ \frac{1}{N} \sum_{j=1}^N \big( F(\underline{x}^{j,x_i})-F(\underline{x})\big)\ ,\label{killclone}
\end{align}
using a change of summation indices in the second term. With the above discussion this can be interpreted as follows: If particle $i$ is less fit than level $c$ it is \textit{killed} and replaced by a uniformly chosen particle $j$, and if it is fitter than $c$ it \textit{cloned}, replacing a uniformly chosen particle $j$.

Observe that, by definition of $\overline{L}^N$ \eqref{overlineLLN}, for any function $F$ on $E^N$ of the form $F(\underline{x})=m(\underline{x})(f)$, with $f\in \CC_b(E)$, we have that
\begin{align}
    \overline{L}^{N} (F)(\underline{x})\,&=\, m(\underline{x})\big(\overline{\LL}_{m(\underline{x})} (f)\big)\ ,\label{IPSgenerator}\\
    \Gamma_{\overline{L}^{N}}( F,\,F)(\underline{x})\,&=\,\frac{1}{N}m(\underline{x})\big(\Gamma_{\overline{\LL}_{m(\underline{x})}}(f,\,f)\big)\ ,\label{IPScarre}
\end{align}
thus conditions \eqref{ass_generator}-\eqref{ass_carre} are satisfied. 

Analogous relations hold also for the individual mutation and cloning parts of the generator. Since generators are linear, the identity \eqref{IPSgenerator} is immediate. The carr\'{e} du champ \eqref{IPScarre} is quadratic in $F$, but off-diagonal terms in the corresponding double sum turn out to vanish in a straightforward computation, leading to the additional factor $1/N$. 
%In general also marginal observables depending only on a single particle state $x_i$ can be of interest, see \cite{angeli2018rare} for a more detailed discussion.
Furthermore, by construction, for almost every realization $\xi_t$, $t>0$, of the mean field particle approximation, there exists at most one particle $i$ such that $\xi^i_t\neq \xi^i_{t-}$, thus condition \eqref{bounded_jumps} is satisfied with $K=1$. Therefore, Theorem \ref{thm_weak} holds and provides $L^p$-error and bias estimates of order $1/\sqrt{N}$ and $1/N$ respectively, in accordance with already established results, e.g. in \cite{rousset2006control,del2003particle,del2013mean}.

\subsection{The Cloning Algorithm}\label{subsection_cloning}
Cloning algorithms have been proposed in the theoretical physics literature \cite{giardina, lecomte2007numerical} for evaluating large deviation functions associated to Markov processes similar to the mean field system \eqref{overlineLLN}, using the same mutation dynamics. While selection and mutation events are independent in the latter due to the additive structure of $\overline{L}^N$ in \eqref{Lmut} and \eqref{Lsel}, in cloning algorithms both are combined to reduce computational cost. We focus the exposition on a variant of the algorithm proposed in \cite{lecomte2007numerical}, but other continuous-time versions can be analysed analogously. This cloning algorithm is constructed from the McKean model $\overline{\LL}_{\mu}$  \eqref{mcKean_generator} with selection rates $\widetilde{W}_c(x,y)= \big(\VV(x)-c\big)^- + \big(\VV(y)-c\big)^+$ as in  \eqref{mcKean1}, and we denote the associated McKean generator by
\begin{equation}\label{lbarmuc}
    \overline{\LL}_{\mu ,c} (f)(x):=\LL (f)(x)+\int_E \widetilde{W}_c (x,y)\big( f(y)-f(x)\big) \mu (dy)\ .
\end{equation}
We will use in particular the killing/cloning interpretation introduced in \eqref{killclone}. We recall that the overall escape rate and probability kernel of the original dynamics $\LL$ are denoted respectively by $\lambda(x)$ and $p(x,dy)$.  

% The dynamics of the cloning algorithm can be described as follows:

% \begin{itemize}
%     \item Any particle $i=1,\dots,N$ makes transitions with rate ${\l}^\star(x_i)$ depending only on the given particle but not on the rest of the population;
%     \item A transition event for a particle $i$ is a combination of mutation and cloning if $\VV(x_i)\geq c$, a combination of mutation and killing otherwise; 
%     \item In case $\VV(x_i)\geq c$, a set $A$ of particles is chosen at random from the ensemble with probability $p_{\underline{x}}(x_i,A)$ and every particle $j\in A$ is replaced by a clone of $i$. Then, the particle $i$ mutates to a new state $y\in E$ with probability kernel $p(x_i,\,dy)$;
%     \item In case $\VV(x_i)<c$, the particle $i$ is killed and replaced uniformly from the population with some probability $p^\star(x_i)$, otherwise the particle $i$ mutates to a new state $y\in E$ with probability kernel $p(x_i,\,dy)$.
% \end{itemize}

% In Proposition \ref{prop_L}, we will show how to choose $\lambda^\star(x_i),\; p^\star(x_i)$ and $p_{\underline{x}}(x_i,A)$ in order to satisfy Assumption \ref{ass_IPS} and, thus, to apply Theorem \ref{thm_weak} and Theorem \ref{thm_CLT}.\\

The infinitesimal description of the cloning algorithm as a continuous-time Markov process on the state space $E^N$ is given by the generator
% \begin{align}
%     \overline{L}^N_c F(\underline{x})\,=\,\sum_{i=1}^N {\l}^\star(x_i)\,\bigg(&\int_E\sum_{A\in \NN}\1_{\{\VV(x_i)\geq c\}}\,p_{x_i}(A)\cdot {p}(x_i,\,dy)\,\big( F(\underline{x}^{A,x_i;\,i,y})-F(\underline{x}) \big)\nonumber\\
%     &+\int_E \1_{\{\VV(x_i)<c\}} \, \big(1-p^\star(x_i)\big)\cdot p(x_i,dy)\,\big( F(\underline{x}^{i,y})-F(\underline{x}) \big)\nonumber \\
%     &+\sum_{j=1}^N \1_{\{\VV(x_i)<c\}} \,\frac{p^\star(x_i)}{N}\, \big( F(\underline{x}^{i,x_j})-F(\underline{x}) \big)\bigg)\ ,
%     \label{genL}
% \end{align}
\begin{align}
    \overline{L}^N_c (F)(\underline{x})\,=&\,\sum_{i=1}^N {\l} (x_i)\,\int_E p(x_i,\,dy)\sum_{A\in \NN}\,\pi_{x_i}(A)\cdot \,\big( F(\underline{x}^{A,x_i;\,i,y})-F(\underline{x}) \big)\nonumber\\
    % (F)(\underline{x})\,=&\,\sum_{i=1}^N {\l} (x_i)\,\int_E p(x_i,\,dy) \bigg( \1_{\{\VV(x_i)> c\}}\sum_{A\in \NN}\,p_{x_i}(A)\cdot \,\big( F(\underline{x}^{A,x_i;\,i,y})-F(\underline{x}) \big)\nonumber\\
    % & \qquad\qquad\qquad\qquad\qquad + \1_{\{\VV(x_i)\leq c\}} \big( F(\underline{x}^{i,y})-F(\underline{x}) \big)\bigg)\nonumber \\
    &+\sum_{i=1}^N \big(\VV(x_i)-c\big)^- \frac{1}{N}\sum_{j=1}^N \big( F(\underline{x}^{i,x_j})-F(\underline{x}) \big) ,
    \label{genL}
\end{align}
for any $F\in\CC_b (E^N)$ and $\underline{x}\in E$. Here $\NN$ is the set of all subsets of $N$ particle indices, $\underline{x}^{A,w}$ denotes the vector $(z_1,\dots,z_N)\in E^N$, with
\begin{align*}
z_j:=
    \begin{cases}
    x_j & j\not\in A\\
    w & j\in A,
    \end{cases}
\end{align*}
and, similarly, $\underline{x}^{A,w;\,i,y}$ denotes the vector $(z_1,\dots,z_N)\in E^N$ with
\begin{align*}
z_j:=
    \begin{cases}
    x_j & j\not\in A,\,j\neq i\\
    w & j\in A,\,j\neq i\\
    y &j=i,
    \end{cases}
\end{align*}
for any $w,y\in E$. 
Cloning events are now coupled with mutation, and if $\VV(x_i)> c$, a non-empty set $A$ of particles is chosen at random from the ensemble with probability $\pi_{x_i}(A)$ and every particle $j\in A$ is replaced by a clone of $i$, before particle $i$ mutates to a new state $y\in E$. If $\VV(x_i)\leq c$ we set $\pi_{x_i} (A)=\delta_{A,\emptyset}$, so that no cloning occurs. Further properties of the cloning distribution $\pi_x(\cdot)$, which is the main distinctive feature of this algorithm, are discussed below. 
The killing part in the second line runs independently and remains unchanged from \eqref{killclone}. The algorithm is often applied in situations with $\VV (x)\geq c$ for all $x\in E$ (in particular also with $c=0$), leaving cloning coupled with mutation as the only selection events.

In order to simplify the presentation, we make some further assumptions on $\pi_{x} (A)$, which are all satisfied by common choices in the theoretical physics literature. 
 %For any $i=1,\dots,N$ and $A\in\NN$ such that $|A|=n$, 
The probability of choosing a set $A$ depends only on its size $|A|$ and not on its elements, i.e. for any $x\in E$ 
 \begin{align}\label{pxA}
 \pi_{x}(A)={\pi_{x,|A|}}\Big/ {\binom{N}{|A|}}\quad &\mbox{with}\quad  \pi_{x,0} ,\ldots ,\pi_{x,N} \mbox{ such that }\sum_{n=0}^N \pi_{x,n}=1\nonumber\\
 & \mbox{and}\quad \pi_{x,n} =\delta_{n,0} \mbox{ if }\VV (x)\leq c\ .    
 \end{align}
 Denote the mean and second moment of this distribution by
\begin{equation}\label{widetilde_Q}
     M(x):=\sum_{n=1}^N n \pi_{x,n},\quad Q(x):=\sum_{n=1}^N n^2 \pi_{x,n}.
\end{equation}
Of course, $\pi_{x,.}$ and its moments also depend on $N$ and $c$, which we omit in the notation for simplicity. In order to ensure that the third condition in Assumption \ref{ass_IPS}, namely \eqref{bounded_jumps}, is satisfied, we assume that the support of $\pi_{x,.}$ is uniformly bounded in $N$, i.e.
\begin{equation}\label{bsup}
    \mbox{there exists $K>0$ such that $\pi_{x,k} =0$ for all $k> K$, $x\in E$}\ .
\end{equation}
 Note that this implies that also $M(x)$ and $Q(x)$ are uniformly bounded, i.e. $M,Q\in\CC_b (E)$. We further assume $N\geq K$, i.e. $N$ is large enough so that the process \eqref{genL} is well defined.

The most common choice in the physics literature (see, e.g., the recent summary in \cite{perez2019sampling}) for the distribution $\pi_{x,.}$ is
\begin{align}\label{common_choice}
    \pi_{x,n}=
    \begin{cases}
    M(x)-\lfloor M(x)\rfloor\quad & n=\lfloor M(x)\rfloor + 1,\\
    \lfloor M(x)\rfloor + 1 - M(x)\quad & n=\lfloor M(x)\rfloor,\\
    0\quad & \mathrm{otherwise}.
    \end{cases}
\end{align}
This corresponds to a binary distribution on the two integers nearest to the prescribed mean, and minimizes the second moment $Q$ of the distribution for a given mean. Note that if $M(x)$ is an integer, $\pi_{x,n} =\delta_{n,M(x)}$ concentrates, which includes the case $M(x)=0$.

%We consider sequences of probabilities $p_{x_i,n}$ that satisfy
%\begin{equation}\label{ass_M}
%    M(x_i)=\frac{\big(\VV (x_i)-c\big)^+}{\widehat{\l}(x_i)},
%\end{equation}
%for any $i=1,\dots,N$.

%\begin{remark}
The next two results assure respectively that condition \eqref{ass_generator} and condition \eqref{ass_carre} in Assumption \ref{ass_IPS} are satisfied for the cloning algorithm, so we can apply Theorem \ref{thm_weak}. The only condition is to choose $M(x)$ such that each particle $i$ produces on average $\big( \VV (x_i )-c\big)^+$ clones per unit time, in accordance with the second term in \eqref{killclone}.
% , for specific choices of $\lambda^\star(x),\; p^\star(x)$ and $M(x)$.

\begin{prop}\label{prop_L}
Consider the cloning generator $\overline{L}^N_c$ \eqref{genL} with
% \begin{equation*}
%     \lambda^\star(x)= \lambda(x)\,+\,\big(\VV(x)-c\big)^-\ , \qquad p^\star(x)=\frac{\big(\VV(x)-c\big)^-}{\lambda(x)\,+\,\big(\VV(x)-c\big)^-}\ ,
% \end{equation*}
% for any $x\in E$, and 
$\pi_{x}(A)$ as in \eqref{pxA} and \eqref{bsup}, such that the mean of the cloning size \eqref{widetilde_Q} is 
\begin{equation}\label{clomean}
    M(x)=\frac{\big(\VV (x)-c\big)^+}{{\l}(x)}\geq 0\quad\mbox{for all }x\in E\ ,
\end{equation}
and $\sup_{x\in E} M(x)<\infty$. 
Then, for any test function of the form $F(\underline{x})=m(\underline{x})(f)$, with $f\in\CC_b (E)$ and $N$ large enough, we get
\begin{align*}
\overline{L}^N_c F(\underline{x})\,= \,m(\underline{x})\big(\overline{\LL}_{m(\underline{x}),c} (f)\big)\ ,
\end{align*}
where $\overline{\LL}_{m(\underline{x}),c} $ is the McKean generator given in \eqref{lbarmuc}.
\end{prop}

\begin{remark}
Note that $\sup_{x\in E} M(x)<\infty$ is essential for \eqref{bsup} and \eqref{bounded_jumps}, and a simple sufficient condition is for the escape rates to be uniformly bounded below, i.e. $\inf_{x\in E} \lambda (x)>0$.
\end{remark}
% \begin{remark}
% If $\VV(x_i)\geq c$, the escape rate is simply $\lambda(x_i)$. The choice of $M(x_i)$ ensures that the average number of clones per unit time originating from particle $i$ is 
% $${\l}(x_i )\, M(x_i )=\big(\VV (x_i )-c\big)^+= \widetilde{W}_c(x_j,x_i)\ ,$$ 
% as required in \eqref{Lclono}. Similarly, if $\VV(x_i)<c$, the average number of killing events for a particle $i$ per unit time is 
% \begin{equation*}
%     \lambda^\star (x_i)\cdot p^\star(x_i) \,=\, \big(\VV (x_i )-c\big)^-= \widetilde{W}_c(x_j,x_i)\ .
% \end{equation*}
% \end{remark}

\begin{proof}
We start by considering the first term in the expression of $\overline{L}^N_c$ \eqref{genL}. Observe that with $F(\underline{x})=m(\underline{x})(f)$,
\begin{align}
    F(\underline{x}^{A,x_i;\,i,y})-F(\underline{x})\,&=\,\frac{1}{N}\big(f(y)-f(x_i)\big)\,+\,\frac{1}{N}\sum_{j\in A} \big( f(x_i)-f(x_j) \big)\nonumber\\
    &=\, \big( F(\underline{x}^{i,y})-F(\underline{x})\big)\,+\, \big(F(\underline{x}^{A,x_i})-F(\underline{x}) \big)\ .\label{deco}
\end{align}
Thus, we can write
\begin{align*}
    \int_{y\in E}&{p}(x_i,\,dy) \sum_{A\in \NN}\pi_{x_i}(A) \,\big( F(\underline{x}^{A,x_i;\,i,y})-F(\underline{x}) \big)\\
    &=\,\int_{y\in E} {p}(x_i,\,dy)\sum_{A\in \NN} \pi_{x_i}(A)\,\bigg( \big( F(\underline{x}^{i,y})-F(\underline{x}) \big) + \big(F(\underline{x}^{A,x_i})-F(\underline{x})\big)\bigg)\\
    &=\, \int_E {p}(x_i,dy)\big( F(\underline{x}^{i,y})-F(\underline{x}) \big) \,+\, \sum_{A\in\NN}\pi_{x_i}(A)\,\big(F(\underline{x}^{A,x_i})-F(\underline{x})\big)\ .
\end{align*}
Moreover, by \eqref{pxA}, we have that, for any $j\in\{1,\dots,N\}$,
\begin{equation}\label{pim1}
    \sum_{{A\in\NN|j\in A}} \pi_{x_i}(A) \,=\, \sum_{n=1}^N \frac{\pi_{x_i,n}}{\binom{N}{n}}\cdot \binom{N-1}{n-1}\,=\, \frac{M(x_i)}{N}\ .
\end{equation}
Therefore, 
\begin{align}
    \sum_{A\in\NN}\pi_{x_i}(A)\,\big(F(\underline{x}^{A,x_i})-F(\underline{x})\big)
    \,&=\,\frac{1}{N}\sum_{j=1}^N M(x_i)\,\big( f(x_i)-f(x_j)\big)\nonumber\\
    &=\, \frac{1}{N}\sum_{j=1}^N \frac{\big(\VV(x_i)-c\big)^+}{\lambda(x_i)}\,\big( f(x_i)-f(x_j)\big)\ .\label{pim2}
\end{align}
Thus, \eqref{genL} can be rewritten as
\begin{align*}
\overline{L}^N_c (F)(\underline{x})\,=&\frac{1}{N}\sum_{i=1}^N {\l} (x_i)\,\int_E p(x_i,\,dy)\big( f(y)-f(x_i )\big)\nonumber\\
&+\frac{1}{N^2}\sum_{i,j=1}^N \Big(\big(\VV(x_j)-c\big)^+ +\big(\VV(x_i)-c\big)^- \Big) \big( f(x_j )-f(x_i ) \big)\nonumber\\
&=m(\underline{x})\big( \overline{\LL}_{m(\underline{x}),c}(f)\big)
    % \sum_{i=1}^N {\l}^\star(x_i)\,&\int_E\sum_{A\in \NN}\1_{\{\VV(x_i)\geq c\}}\,p_{x_i}(A)\cdot {p}(x_i,\,dy)\,\big( F(\underline{x}^{A,x_i;\,i,y})-F(\underline{x}) \big)\\
    % &= \frac{1}{N}\sum_{i=1}^N \1_{\{\VV(x_i)\geq c\}}\,\LL f(x_i)\,+\,\frac{1}{N^2} \sum_{i,j=1}^N {\big(\VV(x_j)-c\big)^+}\, \big( f(x_j)-f(x_i)\big)\ ,
\end{align*}
by changing summation variables in the cloning term and using \eqref{lbarmuc}.

% The second term of expression \eqref{genL} simply gives
% \begin{equation*}
%     \lambda^\star(x_i)\int_E \1_{\{\VV(x_i)<c\}} \, \big(1-p^\star(x_i)\big)\cdot p(x,dy)\,\big( F(\underline{x}^{i,y})-F(\underline{x}) \big)\,=\, \frac{1}{N}\,\1_{\{\VV(x_i)<c\}} \LL f(x_i)\ ,
% \end{equation*}
% because $\lambda^\star(x)\cdot\big(1-p^\star(x)\big)=\lambda(x)$. Finally, the third term in \eqref{genL} can be rewritten as
% \begin{align*}
%     \lambda^\star(x_i)\sum_{j=1}^N \1_{\{\VV(x_i)<c\}}& \,\frac{p^\star(x_i)}{N}\, \big( F(\underline{x}^{i,x_j})-F(\underline{x}) \big)\\
%     &=\,\frac{1}{N^2}\sum_{j=1}^N {\big(\VV(x_i)-c\big)^-}\,\big(f(x_j)-f(x_i)\big)\ .
% \end{align*}
% Combining all together, we obtain the statement.

\end{proof}

\begin{prop}\label{prop_carre}
Let $\overline{L}^N_c$ be a cloning generator satisfying the conditions in Proposition \ref{prop_L}.
% and, without loss of generality, assume that $Q(x)=0$ for every $x\in E$ such that $\VV(x)<c$.
Then, for any test function of the form $F(\underline{x})=m(\underline{x})(f)$, with $f\in\CC_b(E)$,
% and $N$ large enough,
\begin{align*}
\Gamma_{\overline{L}^N_c}(F,\,F)(\underline{x})\,=\,\frac{1}{N}\,m(\underline{x})\Big( G_{m(\cdot)}(f,f)\Big)\, +\, \Ocarre{N}{\underline{x}}{f}\ ,
\end{align*}
as $N\to\infty$, where $\|\Ocarre{N}{.}{f}\|\,\leq\, C\,\frac{\|f\|^2}{N^2}$ 
% \begin{equation*}
%     \sup_{\underline{x}\in E^N}\Ocarre{N}{\underline{x}}{f}\,\leq\, C\,\frac{\|f\|^2}{N^2}\ , 
% \end{equation*}
for some constant $C>0$ independent of $f$ and $N$, and
\begin{align}
    G_\mu (f,f)(x)\,=\,& \Gamma_{\overline{\LL}_{\mu,c}} (f,f) (x)\,+\, \lambda(x)\,\big(Q(x)-M(x)\big)\cdot\big(\ell_{\mu} (f)(x)\big)^2\nonumber\\
    &-\,\frac{2}{\l(x)} \LL (f)(x)\cdot \widetilde{\LL}^{t}_{\mu,c}(f)(x)\ ,\label{G_cloning}
\end{align}
with
% $\Gamma_{\ell_{m(\underline{x})}}$ the carr\'{e} du champ associated to the infinitesimal generator
\begin{equation*}
    \ell_\mu (f)(x):= \int_E \big(f(y)-f(x)\big)\,\mu(dy)\ ,
\end{equation*}
and
\begin{equation*}
    \widetilde{\LL}^{t}_{\mu,c}(f)(x):= \big(\VV (x)-c\big)^+\int_E \big(f(y)-f(x)\big)\,\mu(dy)\ .
\end{equation*}
\end{prop}

\begin{remark}
Due to the linearity of the generator, the combined mutation/cloning events in the cloning algorithm can be decomposed easily, which leads to extra terms only in the quadratic carr\'e du champ. In the expression of the operator $G_\mu$ \eqref{G_cloning}, the term $$\frac{1}{\l(x)} \LL f(x)\cdot \widetilde{\LL}^{t}_{\mu,c}f(x) $$
is due to the dependence between mutation and cloning dynamics and its sign is not known a priori. Whereas, the term $\lambda(x)\,\big(Q(x)-M(x)\big)\cdot\big(\ell_\mu f(x)\big)^2$ arises from the dependence between clones (since multiple cloning events are allowed at the same time) and is always non-negative. In particular, in any setting in which there is at most one clone per event, i.e. when $Q(x)=M(x)$, the term vanishes. Furthermore, minimizing $Q(x)$ as in \eqref{common_choice} for given $M(x)$ \eqref{clomean} leads to the best bound on the carr{\'e} du champ and convergence properties of the algorithm.
\end{remark}

\begin{proof}
Consider the carr\'e du champ of $\overline{L}_c^N$,
\begin{align*}
    \Gamma_{\overline{L}_c^N}&\big(F,\,F \big)(\underline{x})\nonumber\\
    &\,=\,\sum_{i=1}^N \,\bigg(\lambda(x_i)\int_E {p}(x_i,\,dy) \sum_{A\in \NN}\,\pi_{x_i}(A)\cdot \,\big( F(\underline{x}^{A,x_i;\,i,y})-F(\underline{x}) \big)^2\nonumber\\
    % &\qquad +\1_{\{\VV(x_i)<c\}}\lambda(x_i)\int_E  p(x,dy)\,\big( F(\underline{x}^{i,y})-F(\underline{x}) \big)^2\nonumber \nonumber\\
    &\qquad\qquad +\frac{\big(\VV(x_i)-c\big)^-}{N}\sum_{j=1}^N   \big( F(\underline{x}^{i,x_j})-F(\underline{x}) \big)^2\bigg)\ .
\end{align*}
Using \eqref{deco}, the first term can be decomposed as
\begin{align*}
    \int_{ E}&{p}(x_i,dy)\sum_{A\in\NN} \pi_{x_i}(A) \big( F(\underline{x}^{A,x_i;i,y})-F(\underline{x})\big)^2\nonumber\\
    &= \int_{E} p(x_i,dy)\big(F(\underline{x}^{i,y})-F(\underline{x})\big)^2+\sum_{A\in \NN} \pi_{x_i}(A)\big(F(\underline{x}^{A,x_i})-F(\underline{x})\big)^2\nonumber\\
    &+2\int_{ E} {p}(x_i,dy)\ \big(F(\underline{x}^{i,y})-F(\underline{x})\big)\sum_{A\in\NN} \pi_{x_i}(A) \big(F(\underline{x}^{A,x_i})-F(\underline{x})\big)\ ,
\end{align*}
where with \eqref{pim1} and \eqref{pim2} the last line can be rewritten as
\begin{align*}
    % & \l(x_i)\int_{E}\sum_{A\in\NN} p_{x_i}(A)\cdot {p}(x_i,dy)\ \big(F(\underline{x}^{i,y})-F(\underline{x})\big)\cdot \big(F(\underline{x}^{A,x_i})-F(\underline{x})\big)\\
    \frac{2}{N^2}& \int_E p(x_i,dy)\ \big( f(y)-f(x_i)\big) \frac{\big(\VV (x_i)-c\big)^+}{\lambda (x_i)}
    \sum_{j=1}^N\big(f(x_i)-f(x_j)\big)\\
    &=-\frac{2}{N^2}\cdot \frac{1}{\l(x_i)^2} \LL f(x_i)\cdot \widetilde{\LL}^{t}_{m(\underline{x}),c}f(x_i)\ .
\end{align*}
Substituting in the expression of the carr\'{e} du champ $\Gamma_{\overline{L}_c^N} $, we obtain
\begin{align}
    \Gamma_{\overline{L}_c^N}\big(F,\,F \big)(\underline{x})\,=\,& \sum_{i=1}^N\lambda(x_i)\int_{E} p(x_i,dy)\big(F(\underline{x}^{i,y})-F(\underline{x})\big)^2\nonumber\\
    &+\,\sum_{i=1}^N\lambda(x_i)\sum_{A\in \NN} \pi_{x_i}(A)\big(F(\underline{x}^{A,x_i})-F(\underline{x})\big)^2\nonumber \\
    &+ \sum_{i=1}^N  \frac{\big(\VV(x_i)-c\big)^-}{N}\sum_{j=1}^N   \big( F(\underline{x}^{i,x_j})-F(\underline{x}) \big)^2
   \nonumber\\
    &-\,\frac{2}{N^2}\sum_{i=1}^N \frac{1}{\l(x_i)} \LL f(x_i)\cdot \widetilde{\LL}^{t}_{m(\underline{x}),c}f(x_i)\ .\label{carre_expression}
\end{align}
The first line in \eqref{carre_expression} is simply
\begin{equation*}
     \sum_{i=1}^N\l(x_i)\,\int_{E} p(x_i,dy)\big(F(\underline{x}^{i,y})-F(\underline{x})\big)^2\,=\,  \frac{1}{N^2}\,\sum_{i=1}^N\Gamma_\LL (f,f) (x_i)\ .
\end{equation*}
Now, considering the second line of \eqref{carre_expression}, we can write
\begin{align*}
    \l(x_i)&\sum_{A\in \NN} \pi_{x_i}(A)\big(F(\underline{x}^{A,x_i})-F(\underline{x})\big)^2\\
    &=\frac{\l(x_i)}{N^2} \sum_{A\in \NN} \pi_{x_i}(A)\sum_{j,k\in A} \big(f(x_i)-f(x_j)\big)\cdot \big(f(x_i)-f(x_k)\big)\\
    &=\frac{\l(x_i)}{N^2}\bigg( \sum_{j=1}^N \frac{M(x_i)}{N}\,\big(f(x_i)-f(x_j)\big)^2\\
    &\qquad\qquad +\sum_{\substack{j,k=1 \\ k\neq j}}^N \frac{Q(x_i)-M(x_i)}{N(N-1)}\,\big(f(x_i)-f(x_j)\big)\cdot\big(f(x_i)-f(x_k)\big)\bigg)\ ,
\end{align*}
since
\begin{equation*}
    \sum_{A|k,j\in A} \pi_{x_i}(A)\,=\, \sum_{n=2}^N \frac{\pi_{x_i,n}}{\binom{N}{n}}\cdot\binom{N-2}{n-2}\,=\, \frac{Q(x_i)-M(x_i)}{N(N-1)}\ ,
\end{equation*}
for every $j,k\in\{1,\dots,N\}$ such that $j\neq k$.

Recalling that $\lambda(x) M(x)=\big(\VV(x)-c\big)^+$, exchanging summation indices and combining with the third line of \eqref{carre_expression}, we see that
\begin{align*}
    \sum_{i=1}^N&\,\frac{\l(x_i)}{N^2} \sum_{j=1}^N \frac{M(x_i)}{N}\,\big(f(x_i)-f(x_j)\big)^2\,+\\
    &+\, \sum_{i=1}^N  \frac{\big(\VV(x_i)-c\big)^-}{N}\sum_{j=1}^N   \big( F(\underline{x}^{i,x_j})-F(\underline{x}) \big)^2\,=\,\frac{1}{N^2}\sum_{i=1}^N\, \Gamma_{\widetilde{\LL}_{m(\underline{x}),c}}(f,f)(x_i)\ .
\end{align*}

Moreover, 
\begin{align*}
    \sum_{i=1}^N\frac{\l(x_i)}{N^2}&\sum_{\substack{j,k=1 \\ k\neq j}}^N \frac{Q(x_i)-M(x_i)}{N(N-1)}\,\big(f(x_i)-f(x_j)\big)\big(f(x_i)-f(x_k)\big)\\
    &=\,\sum_{i=1}^N\frac{\lambda(x_i)}{N^2} \, \big(Q(x_i)-M(x_i)\big)\, \big(\ell_{m(\underline{x})}f(x_i)\big)^2\,  +\,\Ocarre{N}{\underline{x}}{f}\ .
\end{align*}
with
$$
\Ocarre{N}{\underline{x}}{f} \,=\, \sum_{i,j=1}^N \frac{\l(x_i)\big(Q(x_i)-M(x_i)\big)}{N^3(N-1)}\,\big(f(x_i)-f(x_j)\big)^2\, \leq\, C\,\frac{\|f\|^2}{N^2}\ ,
$$
for all $\underline{x}\in E$, for some constant $C>0$, since $M(x)$ and $Q(x)$ are bounded by condition \eqref{bsup} and $\l(x)$ is bounded by assumption. Combining all together, we obtain the statement.

\end{proof}

Proposition \ref{prop_L} and Proposition \ref{prop_carre} show in particular that Assumption \ref{ass_IPS} is satisfied for cloning algorithms, hence Theorem \ref{thm_weak} holds and provides bias and $L^p$ error bounds.

\section{Proof of Theorem \ref{thm_weak}}\label{section_proofs}

This section is devoted to the proof of Theorem \ref{thm_weak}, which is an adaptation of the results presented by M. Rousset in \cite{rousset2006control}. Throughout this section we consider a generic sequence of IPS generators $(\overline{L}^N)_{N\in\N}$ satisfying Assumption \ref{ass_IPS} for some McKean generator $\overline{\LL}_\mu$ \eqref{mcKean_generator}. Furthermore, we assume that the normalized Feynman-Kac measure $\mu_t$ is asymptotically stable, i.e. Assumption \ref{ass_expstab_generic} holds.

The proof makes use of the propagator $\Theta_{t,T}$ of $\mu_t$ defined in \eqref{normalized_prop}, 
% of $\mu_t$, defined by
% \begin{equation}\label{normalized_prop}
%     \Theta_{t,T}f(x):=\frac{P^{\VV}(T-t)f(x)}{\mu_t\big(P^{\VV}(T-t)1\big)}\,,
% \end{equation}
% for all $f\in\CC_b(E)$, where $P^{\VV}(t)$ is the semigroup associated to the unnormalized marginal $\nu_{t}$ \eqref{nu_eq_general}, i.e. we have $\nu_t(f)=\mu_0(P^\VV(t)f)$ and
% \begin{equation}\label{P_VV}
%     \partial_t \, P^\VV(t)f(x)\,=\,\big(\LL\,+\,\VV\big) P^\VV(t)f(x) \ ,
% \end{equation}
% for any $t\geq 0$ and $f\in\CC_b(E)$, with $P^\VV(0)f=f$.
% Observe that the propagator $\Theta_{t,T}$ satisfies the propagation equation
% \begin{equation*}
%     \mu_T(f)=\mu_t\big(\Theta_{t,T}f\big).
% \end{equation*}
% Furthermore, the proof makes use of 
and the martingale characterization of $\overline{L}^N$.
% Let $(\overline{L}^N)_{N\in\N}$ be a sequence of IPS generators satisfying Assumption \ref{ass_IPS}. 
We denote by $\CC_b^{0,1}(E\times \R^+)$ the set of bounded functions $\varphi_\cdot$ such that $\varphi_t(\cdot)$ is continuous on $E$ for every $t\in\R^+$ and $\varphi_\cdot (x)$ has continuous time derivative for every $x\in E$. Following the standard martingale characterization of Feller-type Markov processes, using It{\^o}'s formula and \eqref{ass_generator} one can show that (see also \cite{rousset2006control}, Proposition 3.3), for every $\varphi_\cdot\in\CC_b^{0,1}(E\times\R^+)$, the process
\begin{equation}\label{martingale}
    \MM^N_t(\varphi_\cdot)\,=\, \mu_t^N(\varphi_t)-\mu_0^N(\varphi_0)-\int_0^t \mu_s^N\big(\partial_s \varphi_s + \overline{\LL}_{\mu_s^N}(\varphi_s)\big)\,ds
\end{equation}
is a local martingale. With \eqref{ass_carre} its predictable quadratic variation is bounded by
\begin{equation}\label{predictable_quadratic_variation}
    \big\langle \MM^N(\varphi_\cdot)\big\rangle_t\,\leq\, \frac{1}{N}\int_0^t \mu_s^N\big(G_{\mu^N_s}(\varphi_s,\,\varphi_s)\big)\,ds\,+\, C\,t\cdot\sup_{s\in[0,t]}\,\frac{\|\varphi_s\|^2}{N^2}\ ,
\end{equation}
for some constant $C\geq 0$ independent of $\varphi$ and $N$, and with \eqref{bounded_jumps} jumps are bounded by
\begin{equation}\label{martingale_jumps}
    \big| \Delta \MM^N_t(\varphi_\cdot) \big|\,\leq\,\frac{2K\,\|\varphi_t\|}{N} \ .
\end{equation}

The following technical Lemma for martingales will play a central role in the proof of Theorem \ref{thm_weak}.

\begin{lemma}\label{lemma6.2}
Let $\MM$ be a locally square-integrable martingale with continuous predictable quadratic variation $\langle\MM\rangle$, $\MM_0=0$ and uniformly bounded jumps $\sup_t |\Delta\MM_t|\leq a<\infty$. Then, for every $q\in \N_0$ and $T\geq 0$, there exists a constant $C_q>0$ such that
\begin{equation*}
    \sup_{t\leq T}\E\big[ \MM_t^{2^{q+1}} \big]\,\leq\, C_q\,\sum_{k=0}^q \,a^{2^{q+1}-2^{k+1}} \E\Big[ (\langle\MM\rangle_T)^{2^k} \Big]\ .
\end{equation*}
\end{lemma}
\begin{proof}
See \cite{rousset2006control}, Lemma 6.2.

\end{proof}

\subsection{Properties of the normalized propagator}

\begin{lemma}\label{diffTheta2n}
For any test function $f\in\CC_b(E)$ and $0\leq t\leq T$, we have for the normalized propagator \eqref{normalized_prop} 
\begin{equation*}
    \partial_t \,\big(\Theta_{t,T}f(x)\big)\,=\, -  \big(\LL\,+\, \VV(x)-\mu_t(\VV) \big)\big(\Theta_{t,T}f(x)\big).
\end{equation*}
\end{lemma}

\begin{proof}
See \cite{rousset2006control}, p. 836. The idea of the proof is to substitute \eqref{nusemi} into the time derivative of $\Theta_{t,T}f$ \eqref{normalized_prop}.

\end{proof}

\begin{lemma}\label{Theta_estimation}
Under Assumption \ref{ass_expstab_generic} on asymptotic stability, for any $0\leq t\leq T$ and $n\in\N$ and $f\in\CC_b(E)$, there exists a constant $\b>0$ such that
\begin{equation*}
    \|\Theta_{t,T}f\|\leq \b\cdot\|f\|\qquad\mathrm{and}\qquad
    \int_t^T \|\Theta_{s,T}f\|^{2^n}ds \leq \b^{2^n}\cdot \|f\|^{2^n}\cdot(T-t).
\end{equation*}
Moreover, for any $\overline{f}:=f-\mu_T(f)$, there exists some $0<\rho<1$, such that
\begin{equation*}
    \|\Theta_{t,T}\overline{f}\|\leq \b\cdot\|\overline{f}\| \cdot\rho^{T-t}\qquad\mathrm{and}\qquad
    \int_t^T \|\Theta_{s,T}\overline{f}\|^{2^n}ds \leq \b^{2^n}\cdot\|\overline{f}\|^{2^n}\,.
\end{equation*}
\end{lemma}

\begin{proof}
The proof can be found in \cite[Lemma 5.1]{rousset2006control} and the result is due to the asymptotic stability of the Feynman-Kac model.

\end{proof}

Observe that, applying Lemma \ref{diffTheta2n} to the martingale characterization \eqref{martingale} of $\overline{L}^N$, we obtain
\begin{align}\nonumber
    \MM^N_T\big( \Theta_{\cdot,T}{f}\big) &= \mu_T^N (f)
    -\mu_0^N\big( \Theta_{0 ,T}f\big)\,-\,\int_0^T \mu_s^N\Big(\big( \widetilde{\LL}_{\mu^N_s}-\VV+\mu_s(\VV)\big)\big(\Theta_{s,T}f\big)\Big)\,ds\\
    \label{M_Theta}
    & =\mu_T^N (f)
    -\mu_0^N\big( \Theta_{0 ,T}f\big)\,-\,\int_0^T \mu_s^N\big(\Theta_{s,T}{f}\big)\cdot\big(\mu_s(\VV)-\mu_s^N(\VV)\big)\, ds\ ,
\end{align}
for any $f\in\CC_b(E)$, where the last equality follows by the characterization \eqref{general_mcKean} of McKean models. 
By \eqref{M_Theta}, we obtain the stochastic differential equation
\begin{equation}\label{diff_eq_with_M}
    d\mu_t^N(\Theta_{t,T}f)=d\MM^N_t(\Theta_{\cdot,T}f)+\big(\mu_t(\VV)-\mu_t^N(\VV)\big)\cdot\mu_t^N(\Theta_{t,T}f)\, dt\ .
\end{equation}
Moreover, applying Lemma \ref{Theta_estimation} to the predictable quadratic variation \eqref{predictable_quadratic_variation}, we obtain that almost surely,
\begin{equation}\label{carre_estimation}
    \big\langle \MM^N(\Theta_{\cdot,T}f)\big\rangle_t\,\leq\, \frac{1}{N}\,\|\overline{G}\|\cdot \beta^2\,\|f\|^2\,(T-t) \,+\, C\,(T-t)\,\frac{\beta^2\cdot \|f\|^2}{N^2}\ ,
\end{equation}
where $ \overline{G}(f,f)=\sup_{\mu\in \PP(E)}G_\mu(f,f)$\ .

Note that Equation \eqref{M_Theta} for centered test functions $\overline{f}=f-\mu_T(f)$ can be rewritten as
\begin{equation}\label{M_Theta_centered}
    \mu^N_T(f)-\mu_T(f)\,=\, \mu^N_0(\Theta_{0,T}\overline{f})\,+\,\MM^N_T(\Theta_{\cdot,T}\overline{f})\,+\,\int_0^T \mu_s^N\big(\Theta_{s,T}\overline{f}\big)\cdot\big(\mu_s(\VV)-\mu_s^N(\VV)\big)\, ds\ .
\end{equation}
The martingale characterization \eqref{M_Theta}-\eqref{M_Theta_centered} will be the key element in the proof of Theorem \ref{thm_weak}.

\subsection{$L^p$ and bias estimates}

Define
\begin{equation}\label{Phi}
    \Phi_{t,T}(\mu):=\frac{\mu P^{\VV}(T-t)}{\mu\big(P^{\VV}(T-t)1\big)}\in\PP(E),
\end{equation}
with $\mu\in\PP(E)$ and $0\leq t\leq T$. Observe that the measure $ \Phi_{t,T}(\mu)$ can be also rewritten in terms of $\Theta_{t,T}$ \eqref{normalized_prop} as
\begin{equation}\label{phi_theta}
    \Phi_{t,T}(\mu)(f)\,=\,\frac{\mu(\Theta_{t,T}f)}{\mu(\Theta_{t,T}1)} \ ,
\end{equation}
for any $f\in\CC_b(E)$. 
To prove Theorem \ref{thm_weak}, we consider the decomposition
\begin{align}
    \E[|\mu_T^N(f)-\mu_T(f)|^p]^{1/p}\nonumber
    \,\leq\; & \E[|\mu^N_T(f)-\Phi_{t,T}(\mu^N_t)(f) |^p]^{1/p}\\&\; +\E[|\Phi_{t,T}(\mu^N_t)(f)-\mu_T(f)|^p]^{1/p},\label{decomposition_atbt}
\end{align}
for any $0\leq t\leq T$. The proof is structured as follows:
\begin{itemize}
    \item In Lemma \ref{lemma_at}, we bound the first term of the decomposition under Assumptions \ref{ass_expstab_generic} and \ref{ass_IPS};
    \item In Lemma \ref{lemma_bt}, we bound the second term under Assumption \ref{ass_expstab_generic};
    \item In Lemma \ref{global_control}, we combine Lemma \ref{lemma_at} and Lemma \ref{lemma_bt} to obtain $L^p$-error estimates of order $1/N^{\delta/2}$, for some $\delta\in (0,1)$;
    \item Finally, from Lemma \ref{global_control} we derive, by iteration, $L^p$ estimates of order $1/\sqrt{N}$, as presented in Theorem \ref{thm_weak}.
\end{itemize}

\begin{lemma}\label{lemma_at}
Consider a sequence of particle approximations satisfying Assumption \ref{ass_IPS} with empirical distributions $\mu_t^N$ \eqref{emea}. Under Assumption \ref{ass_expstab_generic} on asymptotic stability, for any $p\geq 2$ there exists a constant $c_p> 0$ such that
\begin{equation*}
    \E\Big[ \big|\mu_T^N(f)-\Phi_{t,T}(\mu^N_t)(f)\big|^p\Big]\,\leq\, c_p\, e^{4p(T-t)\|\VV\|} \,\bigg(\frac{\|f\|^p\,(T-t)^{p/2}}{N^{p/2}}\bigg)\ ,
\end{equation*}
for any $f\in\CC_b(E)$ and $0\leq t\leq T$.

\end{lemma}

\begin{proof}
This is an adaptation of the first part of the proof of Lemma 5.3 in \cite{rousset2006control}. 
First, consider
\begin{equation}\label{A_t1_t2}
    A_{t_1}^{t_2}:=\exp{\Big(\int_{t_1}^{t_2}\big(\mu_s^N(\VV)-\mu_s(\VV)\big)\,ds\Big)}\ ,
\end{equation}
with $0\leq t_1\leq t_2$. Observe that, by the stochastic differential equation \eqref{diff_eq_with_M}, we can write
\begin{align*}
    d\big(A_t^s \mu_s^N(\Theta_{s,T}f)\big)=A_t^s\,d\MM^N_s(\Theta_{\cdot,T}f)\ ,
\end{align*}
for any $t\leq s\leq T$. Therefore,
\begin{equation}\label{AtT}
    A_t^T \mu_T^N(f)-\mu_t^N(\Theta_{t,T}f)=\int_t^T A_t^s\,d\MM^N_s(\Theta_{\cdot,T}f)\ .
\end{equation}
Fixing $0\leq t\leq T$, the process 
$$ \NN^N_\tau(f):= \int_t^\tau A_t^s \, d\MM^N_s(\Theta_{\cdot,T}f)\,=\, A_t^\tau\cdot \mu_\tau^N\big(\Theta_{\tau,T}f\big)\,-\, \mu^N_t(\Theta_{t,T}f)\ ,$$
with $t\leq \tau \leq T$, as the integral of a progressively measurable process with respect to a local martingale, is itself a local martingale with predictable quadratic variation given by
\begin{equation*}
    \langle\NN^N(f)\rangle_\tau =\int_t^\tau \big(A_t^s\big)^2 d\langle \MM^N_s(\Theta_{\cdot,T}f)\rangle \ ,
\end{equation*}
and jumps bounded by
\begin{equation*}
    \big| \Delta \NN^N_\tau(f)\big|\,\leq\, e^{2(T-t)\|\VV\|} \cdot\frac{4K\,  \beta \,\|f\|}{N} \ ,
\end{equation*}
by Assumption \eqref{bounded_jumps} on bounded jumps, \eqref{martingale_jumps} and Lemma \ref{Theta_estimation}.

Moreover, with \eqref{phi_theta}, we can write
\begin{align*}
    \big|\mu_T^N&(f)-\Phi_{t,T}(\mu^N_t)(f)\big|\\
    &= \Big| \,\mu_T^N(f) -(A_t^T)^{-1}\mu_t^N(\Theta_{t,T}f) - \Big(1-(A_t^T)^{-1}\mu_t^N(\Theta_{t,T}1)\Big)\cdot\Phi_{t,T}(\mu^N_t)(f)\,\Big|\\
    &=(A_t^T)^{-1}\,\Big|\, \NN^N_T(f)\, -\,\NN^N_T(1) \cdot\Phi_{t,T}(\mu^N_t)(f)\,\Big|\ ,
\end{align*}
where the last equality follows by \eqref{AtT}. Noting that $(A_t^T)^{-1}\leq \exp\big(2(T-t)\cdot\|\VV\|\big)$ by definition \eqref{A_t1_t2}, we get
\begin{align}
   \E\Big[ \big|\mu_T^N(f)-\Phi_{t,T}&(\mu^N_t)(f)\big|^p\Big]\nonumber\\
   &\leq\, e^{2p(T-t)\|\VV\|}\, \E\Big[ \,\Big|\, \NN^N_T(f)\, -\,\NN^N_T(1) \cdot\Phi_{t,T}(\mu^N_t)(f)\,\Big|^p\,\Big]\ .\label{at_bound}
\end{align}
By Lemma \ref{lemma6.2}, we have that, for any $q\in \N_0$,
\begin{align*}
    \E\Big[ \big| &\NN^N_T(f)\big|^{2^{q+1}}\Big]\\
    &\leq C_q\,\sum_{k=0}^q \,\Big(e^{2(T-t)\|\VV\|}\cdot\frac{2 K\, \beta \|f\|}{N}\Big)^{2^{q+1}-2^{k+1}} \E\Big[ \big(\big\langle\NN^N_\cdot(f)\big\rangle_T\big)^{2^k} \Big] \\
    &\leq \widetilde{C}_q\,\sum_{k=0}^q \,\Big(e^{2(T-t)\|\VV\|}\cdot\frac{ \|f\|}{N}\Big)^{2^{q+1}-2^{k+1}} \Big( \frac{1}{N}\, \|f\|^2\, (T-t)\Big)^{2^k}\ ,
\end{align*}
where the last inequality follows by \eqref{carre_estimation}. Therefore, for $p=2^{q+1}$, $q\in\N_0$, we get
\begin{equation*}
    \E\Big[ \big| \NN^N_T(f)\big|^{p}\Big]\,\leq\, \widetilde{C}_p\, e^{2p(T-t)\|\VV\|}\, \,\bigg(\frac{\|f\|^p\,(T-t)^{p/2}}{N^{p/2}}\bigg)\ .
\end{equation*}
By Jensen's inequality, this bound holds for any $p\geq 2$. Applying this to inequality \eqref{at_bound}, we obtain the result.

\end{proof}

\begin{lemma}\label{lemma_bt}
Under Assumption \ref{ass_expstab_generic} on asymptotic stability with constants $\alpha >0$ and $\rho\in (0,1)$, we have that for any $p\geq 2$ and any $0\leq t\leq T$ such that $T-t\geq (\log \e\,-\,\log \a)/\log \rho$ for some $\e\in(0,1)$, the following bound holds
\begin{equation*}
    \E\big[|\Phi_{t,T}(\mu^N_t)(f)-\mu_T(f)|^p\big]^{1/p} \leq \frac{4\|f\|\,\a \rho^{T-t}}{1-\e}\ .
\end{equation*}
Furthermore, when $t=0$, there exists a constant $C_p>0$ depending on $p$ such that
\begin{equation*}
    \sup_{T\geq 0}\,\E\big[\big|\Phi_{0,T}(\mu^N_0)(f)\,-\, \mu_T(f)\big|^p\big]^{1/p}\,\leq\,\frac{C_p\,\|f\|}{N^{1/2}}\ .
\end{equation*}

\end{lemma}

\begin{proof}

By definition \eqref{Phi} of $\Phi_{t,T}$, for any $\eta\in\PP(E)$ and $\lambda\in\R$ we have
\begin{align*}
    \Phi_{t,T}(\eta)(f)\,=\, \frac{\eta\big(e^{-(T-t)\lambda}P^\VV (T-t)f\big)}{\eta\big(e^{-(T-t)\lambda}P^\VV (T-t)1\big)}\ .
\end{align*}
Taking $\lambda$ to be the principal eigenvalue of $\LL +\VV$, using Assumption \ref{ass_expstab_generic} on asymptotic stability and the basic fact $\eta(1)=1$, we can write
\begin{align*}
    \eta\big(e^{-(T-t)\lambda}P^\VV (T-t)f\big)\, &\leq\, \mu_\infty(f)\,+\, \|f\|\cdot \a \rho^{T-t}\ ,\\
    \eta\big(e^{-(T-t)\lambda}P^\VV (T-t)1\big)\, &\geq \, 1\,-\,  \a \rho^{T-t}\ .
\end{align*}
Therefore, for $T-t\geq (\log \e\,-\,\log \a)/\log \rho$, for some $\e\in(0,1)$, we have
\begin{align*}
     \Phi_{t,T}(\eta)(f)\,-\, \mu_\infty(f)\,&\leq\, \mu_\infty(f)\cdot\Big(\frac{1}{1-\alpha \rho^{T-t}}\,-\,1\Big)\,+\, \frac{\|f\|\a\rho^{T-t}}{1-\a\rho^{T-t}}\\
     &\leq\,\frac{2\|f\|\,\a \rho^{T-t}}{1-\e}\ ,
\end{align*}
and similarly
\begin{equation*}
    \Phi_{t,T}(\eta)(f)\,-\, \mu_\infty(f)\,\geq\,-\,\frac{2\|f\|\,\a \rho^{T-t}}{1-\e}\ .
\end{equation*}
Therefore,
\begin{align*}
    \E\big[|&\Phi_{t,T}(\mu^N_t)(f)-\mu_T(f)|^p\big]^{1/p}\\
    &\leq \E\big[|\Phi_{t,T}(\mu^N_t)(f)-\mu_\infty(f)|^p\big]^{1/p}+\E\big[|\Phi_{t,T}(\mu_t)(f)-\mu_\infty(f)|^p\big]^{1/p}\\
    &\leq \frac{4\|f\|\,\a \rho^{T-t}}{1-\e}\ .
\end{align*}
Now, for $t=0$, observe that
\begin{align*}
    \Phi_{0,T}&(\mu^N_0)(f)\,-\, \mu_T(f)\\
    &=\,\mu_0^N \big( \Theta_{0,T}(f)\big)-\mu_0\big(\Theta_{0,T}(f)\big)\,+\,\Phi_{0,T}(\mu^N_0)(f) \cdot\big(1-\mu_0^N(\Theta_{0,T}(1)\big)\ .
\end{align*}
Using the basic fact $1=\mu_0(\Theta_{0,T}(1))$, to conclude it is enough to observe that, for any $f\in\CC_b(E)$,
\begin{equation}\label{b0}
    \E\big[ \big| \mu_0^N(f)-\mu_0(f)\big|^p\big] \leq \frac{C_p\,\|f\|^p}{N^{p/2}},
\end{equation}
with $C_p>0$ constant depending on $p$. Indeed, with \eqref{inicon} at time $t=0$, $\mu^N_0 (f)$ is the sum of $N$ i.i.d. random variables with law $f_\# \mu_0$. Inequality \eqref{b0} is then a direct application of Marcinkiewicz-Zygmund/BDG inequalities for i.i.d. variables.
%\marginpar{please check notation for push forward}

\end{proof}

\begin{lemma}\label{global_control}
Consider a sequence of particle approximations satisfying Assumption \ref{ass_IPS} with empirical distributions $\mu_t^N$ \eqref{emea}. Under Assumption \ref{ass_expstab_generic}, there exists $\delta\in(0,1)$ such that for any $p\geq 2$ there exist $c_p>0$ such that
\begin{equation*}
   \sup_{T\geq 0}\, \E[|\mu_T^N(f)-\mu_T(f)|^p]^{1/p}\,\leq\, \frac{c_p\,\|f\|}{N^{\delta/2}}\ ,
\end{equation*}
for any $N\in\N$ large enough.
\end{lemma}

\begin{proof}
Recalling decomposition \eqref{decomposition_atbt}, where the first term is estimated in Lemma \ref{lemma_at} and the second in Lemma \ref{lemma_bt}, and using the basic fact $T-t\leq e^{T-t}$, we obtain
\begin{align}\nonumber
    \E[|&\mu_T^N(f)-\mu_T(f)|^p]^{1/p}\\
    &\leq c_p\,\|f\|\,\cdot \,\frac{e^{(4\|\VV\|+1/2)T}\,+\, 1}{N^{1/2}}\ ,\label{t=0}
\end{align}
taking $t=0$, and
\begin{align}\nonumber
    \E[|&\mu_T^N(f)-\mu_T(f)|^p]^{1/p}\\
    &\leq  c_p\|f\|\cdot\Big(\frac{e^{(4\|\VV\|+1/2)\cdot(T-t)}}{N^{1/2}}\,+\, \rho^{T-t}\Big)\ ,\label{T-t}
\end{align}
taking $0\leq t\leq T$ such that $T-t$ is large enough.

The idea is to find $t\geq 0$ and $\e\in(0,1)$  such that
\begin{align*}
    \begin{cases}
    & \frac{e^{(4\|\VV\|+1/2)\cdot(T-t)}}{N^{1/2}}\,\leq\, \frac{1}{N^{\e/2}}\ ,\\
    & \rho^{T-t}\leq \frac{1}{N^{\e/2}}\ .
    \end{cases}
\end{align*}
Recalling that $\log \rho < 0$, the solution is given by
\begin{align}
    \begin{cases}
    & \e\,=\, \frac{-\log\rho}{4\|\VV\|+\frac{1}{2}-\log\rho}\ ,\\
    & t \,=\, T - \,\frac{\log N}{8\|\VV\|+1-2\log\rho}\ ,\label{t_solution}
    \end{cases}
\end{align}
provided $T\geq {\log N}/({8\|\VV\|+1-2\log\rho})$ to ensure that $t\geq 0$. Also observe that for $N$ large enough, $T-t$ satisfies the conditions in Lemma \ref{lemma_bt}.

Otherwise, in case $T< {\log N}/({8\|\VV\|+1-2\log\rho})$, we consider the bound \eqref{t=0} instead, and we obtain
\begin{equation*}
    \frac{e^{(4\|\VV\|+1/2)T}\,+\, 1}{N^{1/2}}\,\leq\, \frac{1}{N^{\overline{\e}/2}}\,+\,\frac{1}{N^{1/2}}\ ,
\end{equation*}
with
\begin{equation*}
    \overline{\e}\,=\, 1\,-\, \frac{8\|\VV\|+1}{8\|\VV\|+1-2\log\rho}\ .
\end{equation*}
Taking $\delta=\min\{ \e, \,\overline{\e}\}$ the result follows from observing that
\begin{equation*}
    \frac{e^{4(T-t)\|\VV\|}}{N}\,=\,\frac{1}{N^\a}\ ,\quad\mathrm{with}\;\a>\frac{1}{2}\ ,
\end{equation*}
for $t=0$ and $T$ at most of order $\log N$ as above, or for $t\geq 0$ given by \eqref{t_solution}.

\end{proof}

\begin{proof}[Proof of Theorem \ref{thm_weak}]

We denote
\begin{equation*}
    I_p(N)\,:=\, \sup_{\|g\|= 1}\sup_{T\geq 0} \E\big[\big| \mu^N_T(g)\,-\,\mu_T(g)\big|^{p}\big]\ ,
\end{equation*}
in accordance with Rousset \cite{rousset2006control}, Section 5.2. 
Using \eqref{M_Theta_centered}, we have
\begin{align*}
    \big|\mu^N_T(f)-\mu_T(f)\big|^p\,\leq\; & 3^p\,\big|\mu^N_0(\Theta_{0,T}\overline{f})\big|^p\,+\,3^p\,\big|\MM^N_T(\Theta_{\cdot,T}\overline{f})\big|^p\,+\\
    &3^p\,\Big(\int_0^T \big|\mu_s^N\big(\Theta_{s,T}\overline{f}\big)\big|\cdot\big|\mu_s^N(\VV)-\mu_s(\VV)\big|\, ds\Big)^p\ ,
\end{align*}
with $\overline{f}=f-\mu_T (f)$ for any $f\in\CC_b(E)$.

First, observe that, similarly to \eqref{b0}, we have
\begin{equation*}
    \E\big[ \big|\mu^N_0(\Theta_{0,T}\overline{f})\big|^p\big]\,=\,\E\big[ \big|\mu^N_0(\Theta_{0,T}{f})-\mu_0(\Theta_{0,T}{f})\big|^p\big]\,\leq\, \frac{C_p\|f\|^p}{N^{p/2}}\ ,
\end{equation*}
for some constant $C_p>0$ depending on $p$. Moreover, by Lemma \ref{lemma6.2} and bound \eqref{carre_estimation}, we get with another $p$-dependent constant
\begin{equation*}
    \E\big[\big|\MM^N_T(\Theta_{\cdot,T}\overline{f})\big|^p\big] \,\leq\, \frac{C_p\|f\|^p}{N^{p/2}}\ .
\end{equation*}
% for some constant $C_p>0$ depending on $p$.

Finally, writing
\begin{align*}
    \big|\mu_s^N&\big(\Theta_{s,T}\overline{f}\big)\big|\cdot\big|\mu_s^N(\VV)-\mu_s(\VV)\big|\\
    &=\, \|\Theta_{s,T}\overline{f}\|^{1-1/p}\cdot \Big(\,\Big|\mu_s^N\Big(\frac{\Theta_{s,T}\overline{f}}{\|\Theta_{s,T}\overline{f}\|}\Big)\Big|\cdot \|\Theta_{s,T}\overline{f}\|^{1/p}\cdot\big|\mu_s^N(\VV)-\mu_s(\VV)\big|\,\Big) \ ,
\end{align*}
and using H{\"o}lder's inequality, we get
\begin{align*}
    \Big(&\int_0^T  \big|\mu_s^N\big(\Theta_{s,T}\overline{f}\big)\big|\cdot\big|\mu_s^N(\VV)-\mu_s(\VV)\big|\, ds\Big)^p\\
    &\leq\, \Big( \int_0^T \|\Theta_{s,T}\overline{f}\|\,ds\Big)^{p-1}\cdot\Big(\int_0^T \Big|\mu_s^N\Big(\frac{\Theta_{s,T}\overline{f}}{\|\Theta_{s,T}\overline{f}\|}\Big)\Big|^p\cdot \|\Theta_{s,T}\overline{f}\|\cdot\big|\mu_s^N(\VV)-\mu_s(\VV)\big|^p\,ds\Big)\\
    &\leq C_p\|\overline{f}\|^{p-1}\,\Big(\int_0^T \Big|\mu_s^N\Big(\frac{\Theta_{s,T}\overline{f}}{\|\Theta_{s,T}\overline{f}\|}\Big)\Big|^p\cdot \|\Theta_{s,T}\overline{f}\|\cdot\big|\mu_s^N(\VV)-\mu_s(\VV)\big|^p\,ds\Big)\ ,
\end{align*}
by Lemma \ref{Theta_estimation}. Using the fact that
\begin{equation*}
    \mu^N_s\big(\Theta_{s,T}\overline{f}\big)\,=\, \mu^N_s\big(\Theta_{s,T}\overline{f}\big)\,-\, \mu_s\big(\Theta_{s,T}\overline{f}\big) \ ,
\end{equation*}
for centered test functions, and applying the Cauchy-Schwarz inequality, we get
\begin{align}
    \E\bigg[\int_0^T \Big|\mu_s^N\Big(&\frac{\Theta_{s,T}\overline{f}}{\|\Theta_{s,T}\overline{f}\|}\Big)\Big|^p\cdot\big|\mu_s^N(\VV)-\mu_s(\VV)\big|^p\cdot \|\Theta_{s,T}\overline{f}\| \,ds\bigg] \nonumber\\ 
    \leq\,& \int_0^T  \E\Big[\Big|\mu^N_s\Big(\frac{\Theta_{s,T}\overline{f}}{\|\Theta_{s,T}\overline{f}\|}\Big)\,-\,\mu_s\Big(\frac{\Theta_{s,T}\overline{f}}{\|\Theta_{s,T}\overline{f}\|}\Big)\Big|^{2p}\Big]^{1/2}\nonumber\\
    &\cdot\|\VV\|^p\, \E\Big[\Big| \mu^N_s\Big(\frac{\VV}{\|\VV\|}\Big)\,-\,\mu_s\Big(\frac{\VV}{\|\VV\|}\Big)\Big|^{2p}\Big]^{1/2}\cdot  \|\Theta_{s,T}\overline{f}\| \ ds\nonumber\\
    \leq\,& \int_0^T I_{2p}(N)\,\|\VV\|^p\cdot \|\Theta_{s,T}\overline{f}\| \ ds\nonumber\\
    \leq\,& C_p\|f\|\,I_{2p}(N)\ .\label{CS}
\end{align}

Combining all together, we obtain
\begin{equation*}
    \E\big[\big| \mu^N_T(f)\,-\,\mu_T(f)\big|^{p}\big]\,\leq C_p\|f\|^p\,\Big(\frac{1}{N^{p/2}}\,+\, I_{2p}(N)\Big)\ ,
\end{equation*}
for any $f\in\CC_b(E)$ and $T\geq 0$. In particular,
\begin{equation}\label{iteration}
    I_p(N)\,\leq\,  C_p\,\Big(\frac{1}{N^{p/2}}\,+\, I_{2p}(N)\Big)\ ,
\end{equation}
for any $p\geq 2$. Applying Lemma \ref{global_control}, we get
\begin{equation*}
    I_p(N)\,\leq\,
    % \frac{C_p}{N^{\min\{1,2\delta\}p/2}}\,\leq\,
    \frac{C_p}{N^{\min\{1,2^k\delta\}p/2}} \ ,
\end{equation*}
for any $k\in\N$, by iteration of \eqref{iteration}. Thus, we can conclude
\begin{equation*}
    I_p(N)\,\leq\, \frac{C_p}{N^{p/2}}\ .
\end{equation*}
This proves the $L^p$-error estimate \eqref{Lp_estimate}. 

We conclude by proving the bias estimate \eqref{bias_estimate}. By Equation \eqref{M_Theta_centered}, we have
\begin{equation*}
    \E\big[\mu^N_T(f)\big]\,-\,\mu_T(f)\,=\, \int_0^T \|\Theta_{s,T}\overline{f}\|\cdot\E\Big[ \mu^N_s\Big(\frac{\Theta_{s,T}\overline{f}}{\|\Theta_{s,T}\overline{f}\|}\Big)\cdot \big(\mu_s(\VV)\,-\, \mu_s^N(\VV)\big)\Big]\, ds\ .
\end{equation*}
By \eqref{CS} for $p=1$, we obtain
\begin{equation*}
    \big|\E\big[\mu^N_T(f)\big]\,-\,\mu_T(f)\big|\,\leq\, C\|f\|\cdot I_2(N)\,\leq\, \frac{C\|f\|}{N} \ .
\end{equation*}

\end{proof}

\section{Interacting particle approximations for dynamic large deviations}
\label{section3}

\subsection{Large deviations and Feynman-Kac models}

Dynamic large deviations of continuous-time jump processes are a common application area of cloning algorithms \cite{giardina,lecomte2007numerical}. For a given process $( X_t :t\geq 0)$ with bounded rates $W(x,dy)=\lambda (x) p(x,dy)$ \eqref{rates} and path space $\Omega$ as outlined in Section \ref{section1}, we consider a time-additive observable $A_T:\O\to \R$, taken to be a real measurable function of the paths of $X_t$ over the time interval $[0,T]$ of the form \cite{chetrite}
\begin{equation}\label{oss_jump}
    A_T(\o):=\frac{1}{T}\sum_{\substack{t\leq T\\\o(t_-)\neq \o(t)}}g\big(\o(t_-),\,\o(t)\big)\,+\,\frac{1}{T}\int_0^T h\big(\o(t)\big)dt.
\end{equation}
Here $g\in\CC_b(E^2)$ is such that $g(x,x)=0$, for any $x\in E$, and $h\in\CC_b(E)$, with $\o\in\O$ a realization of $(X_t :t\geq 0)$. Note that $A_T$ is well defined since the bound on $\l(x)$ implies that the process does not explode and the first sum contains almost surely only finitely many non-zero terms for any $T\geq 0$.

More precisely, we are interested in studying the limiting behaviour, as $T\to\infty$, of the family of probability measures $\P_{\mu_0}(A_T\in\,\cdot\,)=\P_{\mu_0} \circ A_T^{-1}$ on $(\R,\,\BB(\R))$, where $\mu_0$ represents the initial distribution of the underlying process. This can be characterized by the \textit{large deviation principle} (LDP) \cite{dembo, hollander}, in terms of a \textit{rate function}. We assume that an LDP with convex rate function $I$ holds, which can be written as
\begin{align*}
    &\limsup_{T\to\infty}\,\frac{1}{T} \log \P_{\mu_0}(A_T\in C)\,\leq\,-\,\inf_{a\in C}I(a)\ ,\\
    &\liminf_{T\to\infty}\,\frac{1}{T} \log \P_{\mu_0}(A_T\in O)\,\geq\,-\,\inf_{a\in O}I(a) \ ,
\end{align*}
for every $C\subseteq \R$ closed and $O\subseteq \R$ open. For the study of large deviations, a key role is played by the \textit{scaled cumulant generating function} (SCGF)
\begin{equation}\label{SCGF}
    \L_k:=\lim_{T\to \infty}\frac{1}{T}\log \E_{\mu_0}\big[e^{kTA_T} \big] \in (-\infty ,\infty ] .
\end{equation}
Indeed, if the rate function $I$ is convex and the limit $\L_k$ in \eqref{SCGF} exists and is finite for every $k\in \R$, then $I$ is fully  characterized by the SCGF via Legendre duality (see \cite{dembo}, Theorem 4.5.10), i.e.
\begin{equation*}
    \L_k=\sup_{a \in\R}\{ka-I(a)\}\quad\mbox{and}\quad I(a)=\sup_{k\in \R}\{k\, a-\L_k\}.
\end{equation*}

The SCGF is also the object that can be numerically approximated by cloning algorithms \cite{giardina,lecomte2007numerical} and related approaches and our main aim in this section is to illustrate how our results on Feynman-Kac models can be applied here. 
% In this work, we are interested in providing a
% method for rigorously understanding the convergence properties of such algorithms for evaluating the SCGF based on Feynman-Kac models. 
Possible subtleties regarding the LDP are not our focus and we restrict ourselves to settings where $\L_k$ exists and is finite. In the following we introduce the associated Feynman-Kac models in the notation that is established in this context.

% In order to interpret the SCGF via Feynman-Kac models, we start by introducing an auxiliary infinitesimal generator, known in the physics literature as the \textit{tilted generator} \cite{chetrite}.

\begin{lemma}\label{tilted2}
For any $k \in \mathbb{R}$ the family of operators $\big(P_k(t):t\geq 0\big)$ on $\CC_b(E)$ defined by
\begin{equation}\label{P_k}
    P_{k}(t)f(x)\,:=\, \E_{x}\big[ f\big(X_t\big)\,e^{kt A_t}\big],
\end{equation}
with $f\in \CC_b(E)$, is well defined and it is a non-conservative semigroup, the so-called tilted semigroup. 

Moreover, the infinitesimal generator associated with $\big( P_{k}(t):t\geq 0\big)$, in the sense of the Hille-Yosida Theorem, can be written in the form
\begin{equation}\label{tilt_gen}
\LL_k (f)(x)=\int_E W(x,dy)[e^{kg(x,y)}f(y)-f(x)]\,+\,kh(x)f(x),
\end{equation}
for $f\in \CC_b(E)$ and all $x\in E$, with $g$ and $h$ the bounded continuous functions which characterize $A_T$ via \eqref{oss_jump}. In particular, the semigroup $P_{k}(t)$ satisfies the differential equations
\begin{equation}\label{ddt}
    \frac{d}{dt}P_{k}(t)f\,=\, P_{k}(t)\LL_k (f)\,=\,\LL_k \big(P_k(t)f\big),
\end{equation}
for all $f\in\CC_b(E)$ and $t\geq 0$.
\end{lemma}
\begin{proof}
See \cite{chetrite}, Appendix A.1.

\end{proof}

Observe that, if the SCGF \eqref{SCGF} is independent of the choice of the initial distribution $\mu_0$, it can be written in terms of the tilted semigroup as
\begin{equation}\label{SCGF2}
    \Lambda_k =\lim_{t\to\infty} \frac{1}{t}\log \big( P_k (t)\, 1(x)\big),
\end{equation}
for all $x\in E$, moreover $\Lambda_k$ is the spectral radius of the generator $\LL_k$ (see also \eqref{invariant} below). With Assumption \ref{ass_expstab_generic} on asymptotic stability, $\Lambda_k$ is also the principal eigenvalue of $\LL_k$ and
% Similar to the work in \cite{rousset2006control}, in order to control the asymptotic behaviour of $P_k (t)$, we assume that
there exists a probability measure $\mu_\infty =\mu_{\infty ,k} \in\PP (E)$\footnote{To avoid notation overload, we omit writing explicitly the dependence of certain quantities on the fixed parameter $k$ in the rest of this section.} and constants $\alpha >0$ and $\rho\in(0,1)$ such that
%the semigroup 
%\begin{equation}\label{PLK}
%    P^{\L_k}_k(t):=e^{-t\L_k}\cdot P_k(t)
%\end{equation}
%has stationary measure $\pi\in\MM(E)$ and is asymptotically stable, i.e. there exist constants $\a>0$ and $\rho\in(0,1)$ such that
\begin{equation}\label{asymptotic_stability}
    \big\| e^{-t\L_k} P_k(t)f(\cdot)-\mu_\infty (f)\big\|\leq \|f\|\cdot \a\rho^t,
\end{equation}
for every $t\geq 0$ and $f\in\CC_b (E)$.
%Similar to the work in \cite{rousset2006control}, the asymptotic stability property is essential for studying the convergence properties and sampling errors of the considered particle approximations, as is shown in Section \ref{section3}. 
Note that this implies the independence of the SCGF from the initial distribution, $\mu_0$, and thus \eqref{SCGF2} holds for every initial state $x\in E$. 
Note that \eqref{asymptotic_stability} implies in particular that $\mu_0 e^{-t\L_k} P_k(t)$ converges weakly to $ \mu_\infty$ for all initial distributions $\mu_0$, and that $\mu_\infty$ is the unique invariant probability measure for the modified semigroup $t\mapsto e^{-t\L_k} P_k(t)$. Therefore we have from the generator $\LL_k -\Lambda_k$ of this semigroup that
\begin{equation}\label{invariant}
    \mu_\infty \big(\LL_k (f)\big)=\Lambda_k \mu_\infty (f)\quad\mbox{for all }f\in\CC_b(E)\ .
\end{equation}

Neither the semigroup $P_k (t)$ nor the modified one $e^{-t\Lambda_k} P_k (t)$ conserve probability, and therefore they do not provide a corresponding process to sample from and use standard MCMC methods to estimate the SCGF $\L_k$. This can be achieved by interpreting the tilted generator $\LL_k$ through Feynman-Kac models analogous to Lemma \ref{tilted}, so that we can apply our results from Section \ref{section2}.
% , as illustrated in the following Lemma, so that we can adapt already established results from Feynman-Kac theory \cite{del2013mean, smc:theory:DM00, del2000moran} for approximating $\L_k$.

\begin{lemma}\label{lemma_LLk}
The infinitesimal generator $\LL_k$ \eqref{tilt_gen} can be written as
\begin{equation}\label{LLk}
    \LL_k (f)(x)\,=\,  \widehat{\LL}_k (f)(x) \,+\, \VV_k(x)\cdot f(x),
\end{equation}
 for all $f\in \CC_b(E)$ and $x\in E$. Here 
\begin{equation}\label{widehat}
    \widehat{\LL}_k (f)(x):=\int_E W(x,dy)e^{kg(x,y)} [f(y)-f(x)]
\end{equation}
 is the generator of a pure jump process with modified rates $W(x,dy)\,e^{kg(x,y)}$, and
 \begin{equation}\label{VV}
    \VV_k(x):= \widehat{\l}_k (x)-\l(x)+k h(x)\,\in\,\CC_b(E),
\end{equation}
is a diagonal potential term where $\widehat{\l}_k (x):=\int_{E}W(x,dy)e^{kg(x,y)}$ is the escape rate of $\widehat{\LL}_k$.
\end{lemma}

\begin{proof}
Follows directly from the definition of $\LL_k$ in \eqref{tilt_gen}.
%The proof is given by simply observing
%\begin{align*}
%    \LL_k(f)(x)&=\int_E W(x,dy)[e^{kg(x,y)}f(y)-f(x)]\,+\,kh(x)f(x)\\
%    &=\int_E W(x,dy)e^{kg(x,y)}\cdot[f(y)-f(x)]\,+\,f(x)\cdot\big(\widehat{\l}(x)-\l(x)+kh(x)\big).
%\end{align*}

\end{proof}
In analogy with \eqref{rates}, in the following we also use the notation with a probability kernel
\begin{equation}\label{ratesk}
    W(x,dy) e^{kg(x,y)} = \widehat{\l}_k (x)\,\widehat{p}_k (x,dy)\ .
\end{equation}
Observe that 
\begin{equation}\label{muVk_Lk}
    \LL_k (1)(x)=\VV_k(x),
\end{equation}
thus, we get with \eqref{invariant} another representation of the SCGF,
\begin{equation}\label{SCGF_mu_infty}
\Lambda_k=\mu_\infty (\VV_k )\ .
\end{equation}

Recall the unnormalized and normalized versions of the Feynman-Kac measures defined in \eqref{nu} and \eqref{def_mu} for a given initial distribution $\mu_0 \in\PP (E)$,
\[
\nu_t (f)=\mu_0 \big( P_k (t)f\big)\quad\mbox{and}\quad \mu_t (f)=\nu_t (f)/\nu_t (1)\ ,\quad f\in\CC_b (E)\ ,
\]
and that asymptotic stability \eqref{asymptotic_stability} implies that $\mu_t \to\mu_\infty$ weakly as $t\to\infty$. This suggests the following finite-time approximations for $\Lambda_k$.

\begin{prop}\label{Lk_mu}
For any $k \in \mathbb{R}$ and every $t\geq 0$, we have that
\begin{equation*}
    \log \E_{\mu_0}\big[e^{kt A_t}\big]\,=\, \int_0^t \mu_{s} (\VV_k)\,ds,
\end{equation*}
where $\VV_k$ is defined in \eqref{VV}. In particular, if asymptotic stability \eqref{asymptotic_stability} is satisfied,
%the limit for $T\to\infty$ of
\begin{equation*}
    \frac{1}{T}\int_0^T \mu_{s} (\VV_k)\,ds\,\to\,\L_k \quad\mbox{as }T\to\infty\ .
\end{equation*}
% exists and is finite if and only if the SCGF \eqref{SCGF} is well defined and finite. In that case the limit does not depend on $\mu_0$ and coincides with the SCGF
% \begin{equation}\label{Lk_muV}
% \L_k=\lim_{T\to\infty}\frac{1}{T}\int_0^T \mu_{s,\mu_0} (\VV_k)\,ds.
% \end{equation}%\todo{Should we clarify that this holds for every initial distribution, or indeed that the convergence is uniform over initial distributions?}
\end{prop}

\begin{proof}
Recalling the evolution equation \eqref{nu_eq} of $\nu_{t}$, we have
\begin{equation*}
    \frac{d}{dt} \log \nu_{t}(1)\, =\, \frac{1}{\nu_{t}(1)}\cdot \frac{d}{dt}\nu_{t}(1)\,=\, \frac{\nu_{t}\big(\LL_k( 1)\big)}{\nu_{t}(1)}\,=\,\mu_{t}\big(\LL_k (1)\big).
\end{equation*}
And, thus,
\begin{equation*}
    \nu_{t}(1)\,=\,\exp\bigg( \int_0^t \mu_{s} \big(\LL_k (1)\big)\, ds \bigg),
\end{equation*}
since $\nu_{0}(1)=1$. We can conclude by observing that $\LL_k (1)(x) = \VV_k(x)$ and
\begin{equation}\label{nut1}
   \nu_{t}(1)= \E_{\mu_0}\big[e^{kt A_t}\big],
\end{equation}
using that the SCGF is well defined under asymptotic stability \eqref{asymptotic_stability}.

\end{proof}

% For simplicity, we assume from now on that the limit for $T\to\infty$ of $\L_k^T$ \eqref{LT} exists and is finite.

% \end{proof}

For any $t<T$, we define
\begin{equation}\label{lktt}
    \L_k^{t,T}:=\frac{1}{T-t}\int_{t}^T \mu_s(\VV_k)ds
\end{equation}
as a finite-time approximation for $\L_k$.

\begin{lemma}\label{lemma_conv_Lk}
For any $k\in\R$, under asymptotic stability \eqref{asymptotic_stability} with $\rho\in (0,1)$, there exists a constant $\a'>0$ such that
\begin{equation*}
    \big|\L^{aT,T}(k)-\L(k)\big|\leq \|\VV_k\|\cdot\frac{\a'\,\rho^{aT}}{(1-a)T}\ ,
\end{equation*}
for any given $a\in[0,1)$ and $T> 0$.
\end{lemma}
\begin{proof}
By \eqref{lemma_expstab}, we have
\begin{align*}
    \bigg| \frac{1}{(1-a)T}\int_{aT}^T\mu_t(\VV_k)dt\,-\,\mu_\infty(\VV_k)\bigg|\,&\leq\,\frac{1}{(1-a)T}\int_{aT}^{T}\big|\mu_t(\VV_k)-\mu_\infty(\VV_k) \big| dt\\
    &\leq \frac{1}{(1-a)T}\int_{aT}^T  \|\VV_k\| \cdot{\tilde{\a}\,\rho^{t}}dt\\
    &=\frac{\tilde{\a}\,\|\VV_k\|}{(1-a)T}\cdot\frac{\rho^{T}-\rho^{aT}}{\log \rho}\\
    &\leq\|\VV_k\|\cdot \frac{\a'\,\rho^{aT}}{(1-a)T}\ ,
\end{align*}
where $\a':=\tilde{\a}/(-\log\rho)>0$, using the basic fact $0\leq \rho^{aT}-\rho^T \leq \rho^{aT}$. In particular, $\lim_{T\to\infty}\L^{aT,T}(k)=\mu_\infty(\VV_k)=\L(k)$, by \eqref{SCGF_mu_infty}.

\end{proof}

Note that for $a=0$ the above result only implies a convergence rate of order $1/T$, since errors from the arbitrary initial condition have to be averaged out over time. In contrast for $a>0$ (corresponding to the usual idea of burn-in in conventional Markov chain Monte Carlo approximations -- see \cite{mcmc:tutorial:GS11}, for example), we get a much better exponential rate of convergence dominated by the asymptotic stability parameter $\rho \in (0,1)$.
%\todo{ADD stuff}DEFINE $\Lambda_k^{t,T}$ as later the estimator and add a statement (maybe to the Proposition) of convergence (exponential and $1/T$ if $t=0$) which should follow from spectral gap and/or asymptotic stability?

\subsection{Estimation of the SCGF}\label{4.1}

In this section we establish the convergence of estimators of the SCGF, $\L_k$ \eqref{SCGF}, provided by interacting particle approximations.
%, which is our main result. The error estimates we obtain are based on results in \cite{rousset2006control} for mean field particle approximations (see estimations \eqref{IPS_expectation} and \eqref{IPS_Lp}), that can be extended thanks to the connections between the two models provided in Theorem \ref{L&carre}. Technical details of the adaptation of relevant proofs in \cite{rousset2006control} are postponed to Appendix \ref{appendix_proofs}. 
% For this purpose, we recall that $\Lambda_k$ can be seen as the limit, as $T\to\infty$, of $\Lambda_k^{aT,T}$ \eqref{lktt}, which is the time average of $\mu_t(\VV_k)$ on the time interval $[aT,\,T]$, $a\in(0,1]$, where $\VV_k$ is the potential function given by \eqref{VV} and $\mu_t$ is the normalized Feynman-Kac measure \eqref{def_mu} associated to the tilted generator $\widehat{\LL}_k+\VV_k$.
Approximating $\mu_t$ by the empirical distribution $\mu^N_t$ \eqref{emea} associated to an interacting particle system, we can estimate $\L^{t,T}_k$ with 
\begin{equation}\label{LTN}
    \L_k^{t,T,N}:= \frac{1}{T-t}\int_{t}^T \mu_s^N(\VV_k)\,ds\ .
\end{equation}

Note that, choosing $f\equiv 1$ in Proposition \ref{prop_unbiasedness} and \eqref{nut1} implies that $\exp\big({t\cdot\Lambda_k^{0,t,N}}\big)$ is an unbiased estimator of $\exp\big({t\cdot\Lambda_k^{0,t}}\big)$. Recall that particle approximations are characterized by a sequence of IPS generators $(\overline{L}^N)_{N\in\N}$ on $\CC_b (E^N )$, based on the McKean generators \eqref{mcKean_generator}
\begin{equation*}
    \overline{\LL}_{\mu ,k} :=\widehat{\LL}_k +\widetilde{\LL}_{\mu ,k}\quad\mbox{for all } \mu\in\PP (E)\ ,
\end{equation*}
where $\widetilde{\LL}_{\mu ,k}$ describes the selection dynamics of the McKean model as in Lemma \ref{sufficient_conditions}, with examples in \eqref{mcKean1} or \eqref{mcKean2}. Due to tilted dynamics explained in Lemma \ref{lemma_LLk} we have an additional dependence on the parameter $k$. 

\begin{prop}\label{thm_Lp}
Given $k\in\R$, let $(\overline{L}^N_k)_{N\in\N}$ be a sequence of IPS generators satisfying Assumption \ref{ass_IPS} with McKean generators $\overline{\LL}_{\mu ,k}$. 
% Under Assumption \ref{ass_expstab_generic} on 
Under asymptotic stability \eqref{asymptotic_stability} with $\rho\in (0,1)$, for every $p\geq 2$ and $a\in [0,1)$ there exist constants $c_p, \, c',\,\,\a'>0$ independent of $N$ and $T$, such that
\begin{equation}\label{bound1}
    \E\Big[ \,\big| \L_k^{aT,T,N}-\,\L_k\big|^p\Big]^{1/p}\,\leq\, \frac{c_p}{{N}^{1/2}}\,+\, \frac{\a'\cdot\rho^{aT}}{(1-a)T}\ ,
\end{equation}
and
\begin{equation}\label{bound2}
     \Big|\E\Big[  \L_k^{aT,T,N}\Big]-\,\L_k\Big|\,\leq\, \frac{c'}{N}\,+\, \frac{\a'\cdot\rho^{aT}}{(1-a)T}\ ,
\end{equation}
for any $N\in\N$ large enough and $T>0$.
\end{prop}

\begin{proof}
First, note that
\begin{align*}
    \E\big[ \,\big| \L_k^{aT,T,N}-\,\L_k\big|^p\big]^{1/p}\leq \E\big[ \,\big| \L_k^{aT,T,N}-\,\L_k^{aT,T}\big|^p\big]^{1/p}+\, \,\big|\L_k^{aT,T}-\,\L_k\big|\ .
\end{align*}
The bound for the second term is given in Lemma \ref{lemma_conv_Lk}, whereas we can bound the first term by observing that
\begin{align*}
    \E\Big[ \,\big| \L_k^{aT,T,N}-\,\L_k^{aT,T}\big|^p\Big]^{1/p}
    \leq \frac{1}{(1-a)T} \int_{aT}^T \E\big[ \,\big| \mu^N_t(\VV_k)-\mu_t(\VV_k)\big|^p\big]^{1/p}\, dt\ ,
\end{align*}
and applying Theorem \ref{thm_weak}. The second claim can be established similarly.

\end{proof}

Proposition \ref{thm_Lp} provides the $L^p$ and 
%\todo{This depends how $\tilde{Q}$ depends upon $N$, should we at some point assume that it is uniformly bounded as it is in the example case given below?} 
bias estimates of the approximation error with order of convergence respectively given by $1/\sqrt{N}$ and $1/N$.
%in analogy to \eqref{IPS_Lp} and \eqref{IPS_expectation} for the mean field particle system. The bound depends on the cloning distribution $p_{x} (A)$ only through its second moment ${Q}$ \eqref{widetilde_Q}. 
%The common choice \eqref{common_choice} for the distribution $p_{x,.}$ of the size of cloning events minimizes the second moment $Q$ of the distribution for a given mean and thus its contribution to the bound for the $L^p$ and bias estimates. Even though the bounds are presumably not sharp in practice, this indicates that the above simple choice is reasonable also from a computational perspective.
%In principle, the size distribution $p_{x,.}$ could also be chosen such that ${Q}$ scales with $N$, but our error bounds strongly suggest the use of distributions for which ${Q}$ is uniformly bounded in $N$ as in the above example. In this case, $L^p$ error and bias vanish as $N\to\infty$ with order respectively given by $1/\sqrt{N}$ and $1/N$, which coincide with the rate of convergence for mean field particle approximations given in \eqref{IPS_Lp} and \eqref{IPS_expectation}.
The necessarily finite simulation time $T$ leads to an additional error of order ${\rho^{aT}}/{T}$, with $\rho\in(0,1)$, which is controlled by asymptotic stability properties of the process as summarized in Lemma \ref{lemma_conv_Lk}. 
% As usual, this suggests the use of a ``burn in'' procedure with $a>0$ to obtain an exponential convergence rate with increasing simulation time. 
Ideally, during simulations we want to choose the final
time $T = T(N)$ with respect to the population size $N$ in order to balance both terms in
\eqref{bound1}, resp. \eqref{bound2}. The details depend on asymptotic stability properties of the process and values of constants, but it is clear in general that choosing any $T(N)\gg N$ would only give the same order of convergence as $T(N)\approx N$, which is computationally cheaper. 
Proposition \ref{thm_Lp} also implies that $\L_k^{aT,T,N}$ converges almost surely to $\L_k^{aT,T}$ as $N\to \infty$. 
%For constructing the cloning algorithm to estimate $\Lambda_k^{t,T}$, we consider the McKean representation of $\mu_t$ given by 
%\begin{equation*}
%    \overline{\LL}_{\mu,k} \,=\, \widehat{\LL}_k\,+\, \widetilde{\LL}_{\mu,k} \ ,
%\end{equation*}
%where the non-linear generator $ \widetilde{\LL}_{\mu,k}$ has overall transition rate 
%\begin{equation*}
%    \widetilde{W}_{c,k}(x,y)\,\mu(dy)\,=\,\Big( \big(\VV_k(x) -c\big)^-\,+\, \big(\VV_k(y) -c\big)^+\Big)\,\mu(dy)\ ,
%\end{equation*}
%with $c\in \R$ arbitrary constant. Recall that the escape rate and probability kernel associated to $\widehat{\LL}_k$ are denoted respectively by $\widehat{\lambda}(x)$ and $\widehat{p}(x,dy)$. 
%Consider the cloning generator $\overline{L}^N_{c,k}$ \eqref{genL} with parameters
%\begin{equation*}
%    p(x,dy)=\widehat{p}(x,dy)\ ,\; \lambda^\star(x)= \widehat{\lambda}(x)\,+\,\big(\VV_k(x)-c\big)^-\ , \; p^\star(x)=\frac{\big(\VV_k(x)-c\big)^-}{\widehat{\lambda}(x)\,+\,\big(\VV_k(x)-c\big)^-}\ ,
%\end{equation*}
%for any $x\in E$, and $p_{\underline{x}}(x,A)$ \eqref{pxA} such that the mean of the cloning size \eqref{widetilde_Q} is 
%\begin{equation*}
%    M(x)=\frac{\big(\VV_k (x)-c\big)^+}{{\widehat{\l}}(x)}\quad\mbox{for all }x\in E\ .
%\end{equation*}

\subsection{The Cloning Factor}
Most results in the physics literature do not use the estimator $\L_k^{aT,T,N}$ \eqref{LTN} based on the ergodic average of the mean fitness of the clone ensemble, but an estimator based on a so-called `cloning factor' (see, e.g.,\ \cite{giardina, giardina2011simulating, perez2019sampling}). This is essentially a continuous-time jump process $(C^N_t :t\geq 0)$ on $(0,\infty )$ with $C^N_0=1$, where at each cloning event of size $n\in \N_0\cup\{-1\}$ at a given time $\tau$, the value is updated as
\begin{equation*}
    C^N_t=C^N_{t-}\,\Big(1+\frac{n}{N}\Big)\ ,
\end{equation*}
where $n=-1$ occurs when there is a 'killing' event.
In our context, we can define the dynamics of $C^N_t$ jointly with the cloning algorithm 
%$\zeta_t^N$ 
via an extension of the cloning generator $\overline{L}^N_{c,k}$ \eqref{genL} as introduced in Section \ref{subsection_cloning}, with exit rate $\lambda (x)$ and probability kernel $p(x,dy)$ replaced by $\widehat{\lambda}_k$ and $\widehat{p}_k$, respectively. On the state space $E^N \times (0,\infty)$ define
\begin{align}
    \overline{L}^{(N,\star)}_{c,k} (F^\star )&(\underline{x},\varsigma):=\nonumber\\
    \sum_{i=1}^N \,\bigg(&\widehat{\lambda}_k (x_i)\int_E\widehat{p}_k (x_i,\,dy)\sum_{A\in \NN}\pi_{x_i}(A)\big( F^\star(\underline{x}^{A,x_i;\,i,y},\varsigma_{|A|})-F^\star(\underline{x},\varsigma) \big)\nonumber\\
    %&+\1_{\{\VV_k(x_i)< c\}}\,\widehat{\lambda}(x_i)\int_E  \widehat{p}(x,dy)\,\big( F^\star(\underline{x}^{i,y},\varsigma)-F^\star(\underline{x},\varsigma) \big)\nonumber \\
    &+\sum_{j=1}^N  \,\frac{\big( \VV_k(x_i)-c\big)^-}{N}\, \big( F^\star(\underline{x}^{i,x_j},\varsigma_{-1})-F^\star(\underline{x},\varsigma) \big)\bigg)\ ,\label{cfgen}
\end{align}
where the test function $F^\star :E^N \times (0,\infty)\to\R$ now has a second counting coordinate, and we denote $\varsigma_n :=\varsigma\cdot \big( 1{+}\tfrac{n}{N}\big)$, with $n\in\N_0\cup\{-1\}$. Also recall that the cloning algorithm is based on a McKean model with parameter $c\in\R$ as given in  \eqref{mcKean1}.

We introduce the coordinate projection $G(\underline{x},\,\varsigma):=\varsigma$ in order to observe only the cloning factor, $G(\zeta_t^N,\,C^N_t)=C^N_t$. 
%with $F^\star \in \CC_b(E^N\times \N)$. 
%\todo{$\N$ isn't compact; try to sneak this under carpet for now}
Note that $E^N \times (0,\infty)$ is not compact, and $G$ is an unbounded test function. 
% in general we have to restrict to bounded test functions (see e.g.\ \cite{ethier2009markov}, Section 4.2, p. 162), which is not the case for $G$. %\letizia{how to justify the non-compact state space now?}
However, since the range of the clone size distribution is uniformly bounded (condition \ref{bounded_jumps}), $t\mapsto \log C^N_t$ is a birth-death process on $[0,\infty )$ with bounded jump length, and the generator \eqref{cfgen} and associated semigroup is therefore well defined for the test function $G$ (see e.g. \cite{hamza1995conditions}) and all $t\geq 0$.

The following result provides an unbiased estimator for the unnormalized quantity $\nu_t(1)$ based on the cloning factor.

\begin{prop}\label{unbiasedness}
Let $ \overline{L}_{c,k}^{(N,\star)}$ be the extension \eqref{cfgen} of the cloning generator $\overline{L}_{c,k}^N$ \eqref{genL}. Then, the quantity $e^{tc}C^N_t$ is an unbiased estimator for $\nu_t(1)$ \eqref{nu}, i.e.
\begin{equation*}
    \E\big[ e^{tc}C^N_t \big]\,=\,\E\big[\nu_t^N(1)\big]\,=\, \nu_t(1),
\end{equation*}
for every $t\geq 0$ and $N\geq 1$, and all choices of the parameter $c\in\R$ (cf.\ \eqref{mcKean1}).

\end{prop}

\begin{proof}
First, observe that following \eqref{pim1} and \eqref{pim2}
\begin{align}\label{cl1}
\overline{L}_{c,k}^{(N,\star)} (G)(\underline{x},\varsigma )=&\sum_{i=1}^N \sum_{n=0}^N \widehat{\l}_k (x_i)\, \pi_{x_i,n}\cdot ({\varsigma\,n}/{N})\,-\, \sum_{i=1}^N \frac{\varsigma}{N}\,\big(\VV_k(x_i)-c\big)^-\nonumber\\
=&\,\frac{\varsigma}{N}\,\sum_{i=1}^N {\big(\VV_k (x_i )-c\big)}  \ ,
\end{align}
using the mean $M(x_i)$ of the distribution $\pi_{x_i,n}$ as given in Proposition \ref{prop_L}. Therefore, 
\[
\overline{L}_{c,k}^{(N,\star)} (G)(\underline{x},\varsigma )\,=\, \varsigma\, m (\underline{x}) \big(\VV_k-c\big) \ ,
\]
and analogously to \eqref{vte}, 
the expected time evolution of $C^N_t$ is then given by
\begin{equation*}
    \frac{d}{dt}\E[C^N_t]\,=\,\E[C^N_t\cdot \mu^N_t(\VV_k-c)].
\end{equation*}
This is also the evolution of $\nu_t^N(e^{-tc})=e^{-tc}\nu_t^N(1)$, since
\begin{align*}
    \frac{d}{dt}\E[\nu_t^N(e^{-tc})]\,&=\,\E[\mu_t^N(\VV_k)\cdot e^{-tc}\nu_t^N(1)-c\ e^{-tc}\nu_t^N(1)]\\
    &=\,\E[\nu_t^N(e^{-tc})\cdot \mu^N_t(\VV_k-c)].
\end{align*}
With initial conditions $C_k^N (0)=1=\nu_t^N (1)$, the statement follows by a Gronwall argument analogous to \eqref{gronwall1} and by Proposition \ref{prop_unbiasedness}.

\end{proof}

Proposition \ref{unbiasedness} leads to an alternative estimator for $\L_k^{t,T}$ \eqref{lktt} given by
\begin{equation}\label{cflambda}
    \overline{\Lambda}_k^{t,T,N}:=\frac{1}{T-t}\big(\log C^N_T-\log C^N_t\big)\,+\,c.
\end{equation}
Note that this is not itself unbiased as a consequence of the nonlinear transformation involving the logarithm.

In order to study the convergence of the new estimator to the SCGF, it is convenient to use the martingale characterization of the process, which is given by the following result.

\begin{prop}\label{martingale_factor}
Let $ \overline{L}_{c,k}^{(N,\star)}$ be the extension \eqref{cfgen} of the cloning generator $\overline{L}_{c,k}^N$. Then, the process
\begin{equation*}
    \MM^\star_t:= \log C^N_t-\int_0^t \overline{L}_{c,k}^{(N,\star)} (H)\big(\zeta^N_s,\, C^N_s \big) \, ds,
\end{equation*}
with $H(\underline{x},\varsigma )=\log\varsigma$, is a local martingale satisfying
\begin{equation*}
    \MM^\star_t=\log C^N_t-\int_0^t \big(\mu^N_s(\VV_k)-c\big) \, ds\,+\,t\cdot O\Big(\frac{1}{N}\Big),
\end{equation*}
and with predictable quadratic variation
\begin{equation*}
    \langle \MM^\star_\cdot\rangle_t\,=\, \frac{1}{N}\int_0^t \mu_s^N\big(\widehat{\l}_k\,Q\,+\,(\VV_k-c)^-\big)\,ds\,+\,t\cdot O\Big(\frac{1}{N^2}\Big)\ ,
\end{equation*}
where $Q(x_i)$ is the second moment of the distribution $\pi_{x_i,n}$ \eqref{widetilde_Q}.
\end{prop}

\begin{remark}
Note that, in the case in which there is at most one clone per transition event, i.e. if $Q(x_i)=M(x_i)=(\VV_k(x_i)-c)^+/\widehat{\l}_k(x_i)$, then 
\begin{equation*}
    \langle \MM^\star_\cdot\rangle_t\,=\, \frac{1}{N}\int_0^t \big(\mu^N_s(\VV_k)-c\big)\,ds\,+\,t\cdot O\Big(\frac{1}{N^2}\Big)\ .
\end{equation*}
\end{remark}

\begin{proof}
Observe that we can rewrite \eqref{cfgen} as
\begin{align*}
    \overline{L}_{c,k}^{(N,\star)}(H)(\underline{x},\varsigma)
    % &=& \sum_{i=1}^N\Big( \widehat{\l}_k(x_i)\sum_{A\in\NN}\pi_{x_i}(A)\,\log (1+|A|/N)\,+\,  \big(\VV_k(x_i)-c\big)^-\,\log(1-1/N)\Big)\\
    &=\sum_{i=1}^N \bigg(\sum_{n=0}^N \widehat{\l}_k(x_i)\,\pi_{x_i,n}\log (1{+}{n}/{N})\,+\, \big(\VV_k(x_i){-}c\big)^-\,\log(1{-}1/N)\bigg)\\
    &=m(\underline{x})\big(\VV_k\big)-c\,+\, O\Big(\frac{1}{N}\Big)\ ,
\end{align*}
using the expansion $\log(1+x) =x+O(x^2 )$ as $x\to 0$. Similarly,
\begin{align*}
    \Gamma_{\overline{L}_{c,k}^{(N,\star) }}&(H,H)(\underline{x},\varsigma)\\
    =\,& \sum_{i=1}^N \bigg(\sum_{n=0}^N \widehat{\l}_k(x_i)\, \pi_{x_i,n} \big(\log(1{+}n/N)\big)^2\,+\,  \big(\VV_k(x_i){-}c\big)^-\,\big(\log(1{-}1/N)\big)^2\bigg)\\
    =\,& \frac{1}{N}\, m(\underline{x})\big(\, \widehat{\l}_k Q\,+\, (\VV_k-c)^-\big) \,+\, O\Big(\frac{1}{N^2}\Big).
\end{align*}

The statement corresponds to the martingale problem associated to $\overline{L}_{c,k}^{(N,\star)}(H)$.

\end{proof}

By Proposition \ref{martingale_factor} and recalling the definition of the SCGF estimators $\L_k^{t,T,N}$ \eqref{LTN} and $\overline{\L}_k^{t,T,N}$ \eqref{cflambda} we immediately get
\begin{equation*}
    \L_k^{t,T,N}\,=\,\overline{\L}_k^{t,T,N} \,-\, \frac{\MM^\star_T-\MM^\star_t}{T-t}\,+\,O\Big(\frac{1}{N}\Big)\ .
\end{equation*}

In what follows, we discuss the convergence of the estimator $\overline{\L}_k^{aT,T,N}$ to the SCGF $\L_k$, which is based on the cloning factor.

\begin{thm}\label{thm_cloning}
Let $ \overline{L}_{c,k}^{(N,\star)}$ be the extension \eqref{cfgen} of the cloning generator $\overline{L}_{c,k}^N$. Then, for every $p\geq 2$ and $a\in [0,1)$, there exists a constant $C_p^\star>0$ such that for all $N$ large enough
\begin{equation}\label{th1}
    \E\bigg[ \Big| \overline{\L}_k^{aT,T,N}-\L_k^{aT,T,N}\Big|^p\bigg]^{1/p}\leq\,  \frac{C_p^\star}{N^{1+1/p}\cdot \sqrt{T}}.
\end{equation}
If in addition Assumption \eqref{asymptotic_stability} on asymptotic stability holds, there exist constants $\gamma_p^\star,\,c_p^\star, \a'>0$ and $0<\rho<1$ (dependent on $a,\, p,\,\widehat{\l}_k,\,{Q}$ and $\VV_k$) 
%\todo[color=blue!20]{I included $\widetilde{Q}$ and $\VV_k$ in the constants, otherwise too long} 
such that
\begin{equation*}
    \E\bigg[ \Big| \overline{\L}_k^{aT,T,N}-\L_k\Big|^p\bigg]^{1/p}\leq \frac{\gamma_p^\star}{N^{1+1/p}\cdot \sqrt{T}} \,+\, \frac{c_p^\star}{\sqrt{N}}\,+\,\frac{\a' \rho^{aT}}{T},
\end{equation*}
for every $T\geq 1$.
\end{thm}

\begin{proof}
Thanks to Jensen's inequality, it is enough to prove the inequality for all $p=2^q$, $q\in\N$. First, we can write
\begin{align*}
    \E\bigg[ \bigg| \overline{\L}_k^{aT,T,N}-\L_k^{aT,T,N}\bigg|^{2^q}\bigg]\,&=\,\frac{1}{\big(N\cdot (1-a)T\big)^{2^q}}\cdot \E\Big[\big|\MM^\star_T-\MM^\star_{aT}\big|^{2^q}\Big]\\
    &\leq\,\frac{1}{\big(N\cdot (1-a)T\big)^{2^q}}\cdot \E\Big[\big|\MM^\star_T\big|^{2^q}\Big].
\end{align*}
Observe that $\sup_{t\leq T}\big|\MM^\star_t\big|<\infty$, so the assumptions of Lemma \ref{lemma6.2} are satisfied. Thus, using Lemma \ref{lemma6.2}, we obtain
\begin{align*}
\frac{1}{N^{2^q}\cdot T^{2^q}}\cdot \E\Big[\big|\MM^\star_T\big|^{2^q}\Big]
&\leq\,\frac{C_{q}}{N^{2^q}\cdot T^{2^q}}\,\sum_{k=0}^{q-1} \E\big[ (\langle\MM^\star_\cdot\rangle_T)^{2^k} \,\big]\\
    &\leq\,\frac{\widetilde{C}_q}{N^{2^q}}\,\sum_{k=0}^{q-1}\frac{1}{T^{2^q-2^k}}\,\bigg( \frac{1}{N^{2^k}}\,+\, O\Big(\frac{1}{N^{2^k+1}}\Big)\bigg)\\
    &\leq \, \frac{ C^\star_q}{N^{2^q+1}\cdot T^{2^{q-1}}}\ .
\end{align*}
%\todo{what happened to the sup in t, also in the Lemma $T=\infty$ on the r.h.s.?}
%ALSO: power in $N$ does not match the statement? 
The second part of the Theorem follows directly by Proposition \ref{thm_Lp}.

\end{proof}

Therefore, the $L^p$-error for estimator $\overline{\L}_k^{aT,T,N}$ has the same rate of convergence $1/\sqrt{N}$ as $\L^{aT,T,N}_k$. Analogous results hold for the bias estimates, which have order of convergence $1/N$ as for the estimator $\L_k^{aT,T,N}$ (Proposition \ref{thm_Lp}), since with \eqref{th1} the difference of both estimators is only of order $N^{-1-1/p}$.

\section{Discussion\label{discussion}}

In this work we have established a framework to compare variants of cloning algorithms and understand their connections with mean field particle approximations. 
This allowed us to obtain first rigorous results on the convergence properties of cloning algorithms in continuous time.
Our results apply in the general setting of jump Markov processes on locally compact state spaces. 
%Our approach could be adapted also to locally compact spaces under appropriate boundedness assumptions on the rates using an approach similar to that employed in \cite{rousset2006control}. This may also imply having to restrict the domain of generators and semigroups to continuous functions vanishing at infinity (see e.g. \cite{liggett2010continuous} for a discussion). Since this excludes the constant function $f(x)=1$ which we have made heavy use of, this would require a significant adaption and complication of the presentation of our results. In order to focus attention on the main novelty we have chosen to restrict to compact spaces, which naturally includes the common application cases of these algorithms.
Essential conditions for our approach are summarized in Assumptions \ref{ass_expstab_generic} on asymptotic stability of the process and \ref{ass_IPS} on the particle approximation, which are usually straightforward to check for practical applications. We summarize further sufficient conditions for asymptotic stability in the Appendix \ref{appendix_stability}.

In certain situations the cloning algorithm is computationally cheaper and simpler to implement than mean field particle systems, since only the mutation process has to be sampled independently for all particles and cloning events happen simultaneously. However, as discussed in \cite{angeli2018rare}, this choice reduces in general the accuracy of the estimator since it does not consider the fitness potential of the replaced particles during the cloning events. Adjusting the algorithm by allowing only substitutions of particles with lower fitness based on different McKean models could improve the accuracy. The approach developed in this paper can be used to conduct a systematic study of this question, which is current work in progress.

\appendix

\section{Asymptotic Stability}\label{appendix_stability}

We present sufficient conditions for asymptotic stability as presented in Assumption \ref{ass_expstab_generic}. The discussion is based on the work of Tweedie et al. \cite{down1995exponential,tweedie1994topological}, which we briefly recall in Lemma \ref{tweedie} below.

\begin{defn}
A Feller process $Y_t$ is said to be $\phi$-\textit{irreducible} for a non-trivial measure $\phi$  (i.e. $\phi(E)>0$) on $(E,\BB(E))$, if $\E_x\big[\int_0^\infty \1_{Y_t\in A}dt \big]>0$ for every $x\in E$ and every set $A\in \BB(E)$ such that $\phi(A)>0$. 
%When the measure $\phi$, 
We simply say that $Y_t$ is \textit{irreducible} if it is $\phi$-irreducible for some $\phi$.
\end{defn}

\begin{defn}
A $\phi$-irreducible Feller process $Y_t$ is called \textit{aperiodic} if there exists a small set $C\in \BB(E)$, $\phi(C)>0$, such that the associated Markov semigroup $P(t)$ satisfies the following conditions:
\begin{itemize}
    \item there exists a non-trivial measure $\eta$ and $t>0$ such that $P(t)\,(x,B)\geq \eta(B)$, for all $x\in C$ and $B\in\BB(E)$;
    \item there exists a time $\tau\geq 0$ such that $P(t)\,(x,C)>0$, for all $t\geq \tau$ and $x\in C$.
\end{itemize}
\end{defn}

%NOTE: under aperiodicity, any petite set is small 

\begin{lemma}\label{tweedie}
Let $Y_t$ be a $\phi$-irreducible and aperiodic Feller process on a locally compact state space $E$ such that $\mathrm{supp}\,\phi$ has non-empty interior. Denote by $\LL$ and $P(t)$ the associated infinitesimal generator and the semigroup, respectively. Assume that for a given function $h\in\CC_b(E)$ such that $h\geq 1$, there exist constants $b,c>0$ and a compact set $S\in\BB(E)$ such that for all $x\in E$
\begin{equation*}
    \LL (h)(x)\leq -c\cdot h(x)+b\1_S (x) . 
\end{equation*}
Then there exist constants $\a\geq 0$ and $\rho\in (0,1)$ such that for any test function $f\in\CC_b(E)$ and $t\geq 0$, 
\begin{equation*}
    \big| P(t)f(x)-\pi(f)\big|\leq \| f\|\,h(x)\cdot \a\rho^t,
\end{equation*}
for any $x\in E$, where $\pi$ is the (unique) invariant measure of $Y_t$.
\end{lemma}

\begin{proof}
See \cite{down1995exponential}, Theorem 5.2(c), using the fact that if a Feller process $Y_t$ is $\phi$-irreducible and $\mathrm{supp}\,\phi$ has non-empty interior, then every compact set is petite (See \cite{tweedie1994topological}, Theorem 7.1 and Theorem 5.1).

\end{proof}

In the following we discuss how the spectral properties of the tilted generator $\LL^\VV$ in Assumption \ref{ass_expstab_generic} can imply asymptotic stability in the sense of \eqref{asssta}. 

\begin{assumption}\label{Assumption1}
We assume that the spectrum of $\LL^\VV$ \eqref{fkgen} is bounded by a greatest eigenvalue $\lambda_0$. Moreover, there exist a positive function $r\in \CC_b(E)$, unique up to multiplicative constants, and a probability measure $\mu_\infty \in \PP(E)$ satisfying respectively
\begin{equation*}
    \LL^\VV (r) =\lambda_0\cdot r\ ,
\end{equation*}
and
\begin{equation*}
    \mu_\infty \big(\LL^\VV (f)\big)=\lambda_0\cdot\mu_\infty (f)\quad\mbox{for any }f\in\CC (E)\ .
\end{equation*}
Without loss of generality, we can assume $\mu_\infty (r)=1$.
\end{assumption}

\begin{remark}
Sufficient conditions for Assumption \ref{Assumption1} to hold can be found, for instance, in \cite{gong2006spectral, gong2001poincare}. These are of course satisfied if the original process with generator $\LL$ is an irreducible, finite-state Markov chain, including for example stochastic particle systems on finite lattices with a fixed number of particles.
\end{remark}

%Recall that the SCGF $\Lambda_k$ \eqref{SCGF} is given by the spectral radius of the tilted generator $\LL_k$ \eqref{tilt_gen}.

Under Assumption \ref{Assumption1}, we define the generator
\begin{equation*}
    \LL^\VV_r(f)(x) \,=\, r^{-1}(x)\cdot \LL^\VV (r\cdot f)(x)\,-\, \lambda_0\cdot f(x)\ ,
\end{equation*}
which is known in the literature as \textit{Doob's $h$-transform} of $\LL^\VV$ \citep{chetrite} or twisted Markov kernel \citep{whiteley2017calculating}. Observe that $\LL^\VV_r1=0$, so that it is a probability generator associated to a Markov process with probability semigroup defined for any $f \in \CC_b(E)$ by
\begin{equation*}
    P^\VV_r(t)f(x):= r^{-1}(x)\cdot e^{-\lambda_0}\,P^\VV (t) (r f)(x).
\end{equation*}

\begin{prop} [Asymptotic stability]\label{prop_asymptotic_stability}
Assume that there exists $\e>0$ such that the set
\begin{equation*}
    K_\e:=\big\{x\in E\,\big|\, \VV(x)\geq \lambda_0-\e\big\}
\end{equation*}
is compact. Under Assumption \ref{Assumption1}, if the initial pure jump process $(X_t :t\geq 0)$ with generator $\LL$ is $\phi$-irreducible for some $\phi$ for which $\textrm{supp}\,\phi$ has non-empty interior, and aperiodic as defined above then \eqref{asssta} holds, i.e. there exists $\alpha >0$ and $\rho\in (0,1)$ such that
\begin{equation*}
    \big\| e^{-\lambda_0}P^\VV(t)f-\mu_\infty (f)\big\|\leq \|f\|\cdot \a\rho^t
\end{equation*}
for every $t\geq 0$ and $f\in\CC_b (E)$.
\end{prop}

\begin{proof}
First, note that if the initial process $X_t$ is irreducible and aperiodic, then also the process associated to $\LL^\VV_r$ is irreducible and aperiodic. Moreover, $\LL^\VV_r$ is bounded in ${K_\e}$ and $\LL^\VV_r(r^{-1})\leq -\e\,r^{-1}$ for every $x\not\in K_\e$. Therefore, the hypotheses of Lemma \ref{tweedie} are satisfied for the generator $\LL^\VV_r$ acting on the function $h=r^{-1}$. Thus, applying the lemma we obtain
\begin{equation*}
    \big| P^\VV_r(t)f(x)-\pi(f)\big|\leq \|f\|\,r^{-1}(x)\cdot \a \rho^t,
\end{equation*}
for any $f\in\CC_b(E)$ and $x\in E$, where $\pi(\cdot)=\mu_\infty (r \,\cdot)\in \PP(E)$ is the invariant measure for $\LL^\VV_r$. Dividing by $r^{-1}(x)$ and substituting $f$ with $r^{-1}f\in\CC_b(E)$, we obtain the statement ($\| r^{-1}\| <\infty$ and can be included in the constant $\a$).

\end{proof}

% Proposition \ref{prop_asymptotic_stability} assures asymptotic stability as formulated in Assumption \ref{ass_expstab_generic} and, in particular, we can see that $\mu_\infty =\ell$.

%\begin{cor}\label{cor_as}
%Under the assumptions of Proposition \ref{prop_asymptotic_stability} (equivalently, under Assumption \ref{ass_expstab}), there exist constants $\a\geq 0$ and $0<\lambda<1$ such that for any probability measure $\eta\in\PP(E)$ and function $f\in\CC_b(E)$,
%\begin{equation*}
%    \big|\mu_{t,\mu_0}(f)-l_k(f)\big|\leq \|f\|\cdot {\a\lambda^{t}},
%\end{equation*}
%for all initial distributions $\mu_0\in\PP(E)$.
%\end{cor}
%\begin{proof}
%It is a direct consequence of asymptotic stability. See \cite{rousset2006control}, Corollary 2.3, for details.

%\end{proof}

\section*{Acknowledgments}
{This work was supported by The Alan Turing Institute under the EPSRC grant EP/N510129/1 and
the Lloyd's Register Foundation--Alan Turing Institute Programme on Data-Centric Engineering; AMJ was partially supported by EPSRC grants EP/R034710/1 and EP/T004134/1.}

\bibliographystyle{unsrt}
\bibliography{ms}

\end{document}